\RequirePackage{fix-cm}
\documentclass[smallextended]{svjour3}       
\smartqed  
\usepackage{graphicx}
\usepackage{amssymb}
\usepackage{amsmath}
\usepackage{mathtools}
\usepackage{graphicx, tikz-cd}
\usepackage[]{xcolor}
\usepackage{verbatim}
\usepackage{shuffle}
\usepackage{enumitem}
\usepackage{dirtytalk}
\usepackage{ mathrsfs }

\let\oldparagraph=\paragraph
\renewcommand\paragraph[1]{\oldparagraph{#1.}}
\let\oldexample=\example
\renewcommand\example[1]{\oldexample{. #1}}
\let\oldremark=\remark
\renewcommand\remark[1]{\oldremark{. #1}}

\newcommand{\X}{\mathbb{X}}
\newcommand{\Y}{\mathbb{Y}}

\newcommand{\cX}{\mathcal{X}}
\newcommand{\cL}{\mathcal{L}}

\newcommand{\cH}{\mathcal{H}}

\newcommand{\bN}{\mathbb{N}}
\newcommand{\bR}{\mathbb{R}}
\newcommand{\bX}{\mathbb{X}}
\newcommand{\bfX}{\mathbf{X}}
\newcommand{\bfY}{\mathbf{Y}}

\newcommand{\R}{\mathbb{R}}
\newcommand{\D}{\, \mathrm{d}}
\newcommand{\bP}{\mathbb{P}}
\newcommand{\bQ}{\mathbb{Q}}

\newcommand{\E}{\mathbb{E}}

\newcommand{\W}{\mathbb{W}}
\newcommand{\norm}[1]{\left\lVert#1\right\rVert}
\newcommand{\sprod}[2]{\langle #1, #2 \rangle}

\newcommand{\cS}{\mathcal{S}}

\usepackage[square,sort,comma,numbers]{natbib}

\begin{document}

\title{Rough kernel hedging 
}


\author{Nicola Muça Cirone \and
        Cristopher Salvi 
}


\institute{N. Muça Cirone \at
              Department of Mathematics \\
              Imperial College London\\
              \email{n.muca-cirone22@imperial.ac.uk} 
           \and
           C. Salvi \at
           Department of Mathematics and Imperial-X \\
              Imperial College London\\
              \email{c.salvi@imperial.ac.uk}
}

\date{}

\maketitle

\begin{abstract}
Building on the functional-analytic framework of operator-valued kernels and un-truncated signature kernels \cite{salvi2021signature}, we propose a scalable, provably convergent signature-based algorithm for a broad class of high-dimensional, path-dependent hedging problems. We make minimal assumptions on market dynamics by modelling them as general geometric rough paths, yielding a fully model-free approach. Moreover, by means of a representer theorem, we provide theoretical guarantees on the existence and uniqueness of a global minimum of the resulting optimization problem, and derive an analytic solution under highly general loss functions. Similar to the popular deep hedging \cite{buehler2019deep}—but in a more rigorous fashion—our method can also incorporate additional features by means of the underlying operator-valued kernel, such as trading signals, news analytics, and past hedging decisions, aligning closely with true machine-learning practice.

\end{abstract}

\section{Introduction}\label{intro}

In idealized, complete, and frictionless markets, it is theoretically possible to perfectly hedge financial derivatives, thereby eliminating risk through appropriate hedging strategies. However, real markets are incomplete due to transaction costs, market impact, liquidity constraints, and other frictions, making perfect hedging generally unattainable. Consequently, traders aim to find approximate hedging strategies that minimize a cost function selected according to their risk preferences.

There exists an extensive body of literature on hedging strategies in market models that include frictions, highlighting the complex nature of the problem. In markets where trading activities temporarily impact asset prices, \cite{rogers2010cost} analyses a one-dimensional Black-Scholes model where trading a security affects its price. The optimal trading strategy in this setting is derived by solving a system of three coupled nonlinear PDEs. Extending this framework, \cite{bank2017hedging} addresses a more general tracking problem, which encompasses the temporary price impact hedging issue, within a Bachelier model. The authors provide a closed-form solution for the strategy that involves conditional expectations of a time integral over the optimal frictionless hedging strategy. In the context of proportional transaction costs, \cite{soner1995there} shows that in a Black-Scholes market, the minimal superhedging price for a European call option is equal to the current spot price of the underlying asset. This result implies that super-replication strategies are of limited practical interest in one-dimensional cases due to their prohibitive cost and become numerically intractable in higher-dimensional settings. 

Mean-variance hedging has been another area of focus in the literature. \cite{hubalek2006variance} tackles the mean-variance hedging problem for vanilla options on Lévy processes, offering a semi-explicit solution. Their approach is further developed in \cite{goutte2014variance}, whose authors present an algorithm for the mean variance hedging of vanilla options. In a frictionless framework, \cite{jeanblanc2012mean} uses BSDEs to derive the optimal hedge for the mean-variance problem. However, finding the optimal hedge remains challenging in practical applications. 

Machine learning techniques have also been applied to hedging problems. \cite{hutchinson1994nonparametric} employs neural networks to price and hedge European options. Building on this, \cite{buehler2019deep} extends the methodology to exotic derivatives by approximating the optimal hedging strategy using deep learning. Despite its potential, this approach generally finds only local minima rather than global ones, and the computational cost of training can be significant. 

An alternative method involves the use of truncated signatures from the rough path theory \cite{lyons1998differential} to linearize the general hedging problem. The motivation for using signatures in the context of path-dependent financial optimization problems stems from the well-known fact that linear functionals acting on the signature form a unital algebra of real-valued functions on paths that separates points \cite{lyons2004differential}; thus, by the classical Stone-Weierstrass theorem, such linear functionals of signatures are universal approximators for functions on (compact sets of unparameterized) paths. 
Hence, signature coefficients form an ideal set of features for machine learning applications dealing with sequential data \cite{fermanian2023new, salvi2023structure, cass2024lecture}. In fact, signature methods have become very popular in recent years in quantitative finance \cite{arribas2020sigsdes, salvi2021higher, horvath2023optimal, pannier2024path} as well as in the broader data science communities, and have been applied in a variety of contexts including deep learning
\cite{kidger2019deep, morrill2021neural, cirone2023neural, hoglund2023neural, cirone2024theoretical, issa2024non, barancikova2024sigdiffusions}, kernel methods \cite{salvi2021signature, lemercier2021distribution, lemercier2021siggpde, manten2024signature}, cybersecurity \cite{cochrane2021sk}, and computational neuroscience \cite{holberg2024exact}.
Going back to the hedging problem, \cite{lyons2020non} proposes to use this approach, but it requires several stringent assumptions: 1) payoffs are approximated by signature payoffs; 2) the method is limited to low-dimensional market dynamics due to the exponential growth of signature features; 3) the theory is developed for one-dimensional market dynamics and assumes that the augmented process with its quadratic covariation (i.e., lead-lag transformation) forms a geometric rough path—a property that holds almost surely for semimartingales as shown in \cite{flint2016discretely}, but not for general rough paths; 4) it is restricted to polynomial or exponential hedging losses.

\paragraph{Contributions} In this paper, we utilize the well-established functional-analytic framework of operator-valued kernels together with untruncated signature kernels \cite{salvi2021signature} to provide a scalable and provably convergent signature-based algorithm for a broad class of path-dependent hedging problems that resolves all the limitations mentioned above in the previous work and is applicable to high-dimensional financial data streams. We impose minimal assumptions on the market dynamics by modelling them as general geometric rough paths. Our approach is thus entirely model-free. We present a novel representer theorem that guarantees the existence and uniqueness of a global minimum, and provide an analytic solution to a broad class of hedging problems with highly general loss functions. 
Similar to deep hedging, but in a more rigorous manner, our method considers not only the prices of hedging instruments, but can also incorporate features captured by the operator-valued kernel. 
These features can include additional information like trading signals, news analytics, or past hedging decisions—quantitative data that a human trader might utilize—in line with true machine learning practices. 

\paragraph{Outline of the paper}

In Section \ref{sec:quadratic_example}, we summarize the theoretical foundations and numerical implementation of our method in the context of a simple quadratic hedging problem which we use as a running example throughout the paper. In Section \ref{sect:signatures}, we introduce signature kernels within our framework and explain how this class of kernels provide the necessary guarantees to solve the hedging problem. In Section \ref{sec:general_theo}, we develop a general and flexible theory for kernel hedging, going well beyond the quadratic case. Finally, in Section \ref{sec:experiments}, we present simple experimental results that validate our theoretical findings.

\section{Motivating Example - Quadratic Hedging}\label{sec:quadratic_example}

For demonstration purposes, we summarise the key ideas of our method for solving a quadratic hedging problem.
We assume familiarity of the reader with the theory of rough paths; a brief primer on the subject is given in Appendix \ref{app:RP_Primer}. We refer to \cite{friz2020course, lyons2002system} for an in-depth treatment of the topic.

\medskip


Let $V = \R^d$ be the space in which the market prices $X$ take values and fix a temporal horizon $T > 0$. 
Since the profit-and-loss (P\&L) of a trading strategy is given by an integral against $X$ we need to fix, a priori, a good notion of integration. 
To this end, we assume the prices are directly given as $\alpha$-H\"older rough paths.
In particular, for some initial value $x \in V$, we assume $\mathbf{X} \in \mathscr{C}^{\alpha}_x([0,T], V)$ with first-level increments $\mathbf{X}^0 = X$ matching the observed prices $X$, and where for any $t \in [0,T], ~ ~ \mathscr{C}^{\alpha}_x([0,t], V) =: \Omega^{\alpha}_{x; t} = \Omega^{\alpha}_{x; t}(V)$ denotes the space of $\alpha$-H\"older rough paths starting in $x$.
Define the set of market histories as the following Polish space
$$
\Lambda^\alpha_x = \Lambda^\alpha_x(V) := \bigcup_{t \in [0,T]} \Omega_{x;t}^\alpha(V).
$$

Consider then fixed a probability measure $\mathbb{Q}$ on the measurable space $(\Omega^\alpha_{x;T}, \mathcal{B}(\Omega^\alpha_{x;T}))$ where $\mathcal{B}(\Omega^\alpha_{x;T})$ is the Borel $\sigma$-algebra and let $\mathbb{F} = \{\mathcal{F}_t : t \in [0,T]\}$ be the filtration generated by the process of (lifted) market prices $\mathbf{X}$. 
We will work on the filtered space $(\Omega^\alpha_{x;T}, \mathcal{B}(\Omega^\alpha_{x;T}), \mathbb{F}, \mathbb{Q})$.
\\

We are interested in solving the following \emph{quadratic hedging problem}
\begin{equation}\label{eqn:hedging}
    \inf_{f \in \mathfrak{F}} \E^x_\bQ\left[\left( \pi(\mathbf{X}|_{[0,T]})  - \pi_0 - \int_0^T \left\langle f(\mathbf{X}|_{[0,t]}), \D \mathbf{X}_t \right\rangle_{V}\right)^2 \right],
\end{equation}
where $\mathfrak{F} = \{f : \Lambda^\alpha_x \to V \}$ is a family of functionals generating predictable, self-financing strategies $t \mapsto f(\mathbf{X}|_{[0,t]})$ integrable against $\D \bfX$, $\pi : \Omega_{x;T}^\alpha \to \bR$ is a (possibly path-dependent) payoff and $\pi_0\in\bR$ the initial capital, for some starting point $x \in V$. To ease notation we remove the dependence on $x$. 
\\

We will consider our family $\mathfrak{F}$ to be a Reproducing Kernel Hilbert Space (RKHS).
To define these spaces in our vector-valued setting we need the notion of \emph{operator-valued kernel}:
\begin{definition}(Operator Valued Kernel)
    Given a set $\mathcal{Z}$ and a real Hilbert space $\mathcal{Y}$, an $\mathcal{L}(\mathcal{Y})$-valued kernel is a map $K: \mathcal{Z} \times   \mathcal{Z} \to \mathcal{L}(\mathcal{Y})$ such that
    \begin{enumerate}[label=\roman*.]
        \item $\forall x,z \in \mathcal{Z}$ it holds that $K(x,z) = K(z,x)^T$.
        \item For every $N \in \bN$ and $\{(z_i,y_i)\}_{i=1,\dots,N}\subseteq \mathcal{Z} \times \mathcal{Y}$ the matrix with entries $\sprod{K(z_i,z_j)y_i}{y_j}_{\mathcal{Y}}$ is semi-positive definite.
    \end{enumerate}
\end{definition}
A classical result  \cite{carmeli2006}  associates to such a kernel $K$ a particular Hilbert space $\cH_{K} \subseteq \{\mathcal{Z} \to \mathcal{Y}\}$, its reproducting kernel Hilbert space (RKHS), characterised by the following properties:
\begin{enumerate}[label=\roman*.]
    \item $\forall z,y \in \mathcal{Z} \times \mathcal{Y}$ it holds $K(z,\cdot)y \in \cH_{K}$.
    \item $\forall z,y \in \mathcal{Z} \times \mathcal{Y}$ and $\forall \phi \in \cH_{K}$ it holds $\sprod{\phi}{K(z,\cdot)y}_{\cH_{K}} = \sprod{\phi(z)}{y}_{\mathcal{Y}}$.
    \item $\cH_{K}$ is the closure, under $\sprod{\cdot}{\cdot}_{\cH_{K}}$ of $\text{Span}\left\{ K(z,\cdot)y ~|~  z,y \in \mathcal{Z} \times \mathcal{Y} \right\}$.
\end{enumerate}


Consider then an operator-valued kernel 
$K : \Lambda^\alpha \times \Lambda^\alpha \to \cL(V)$ and denote by $\cH_{K}$ the associated RKHS. The elements of $\cH_{K}$ can be understood as maps $\Lambda^\alpha \to V^* \simeq V$ via the defining reproducing property \emph{i.e.} for $(f, \bfX, v) \in \cH_{K} \times \Lambda^\alpha \times V$ 
\begin{equation}\label{eqn:rep_property}
    \sprod{f(\bfX)}{v}_V = \sprod{f}{K(\bfX,\cdot)v}_{\cH_{K}}.
\end{equation}
These are exactly the types of families $\mathfrak{F}$ of functionals which we are going to consider in this work.

\begin{remark}
    In practice we might want $\mathfrak{F}$ to be a general class of functionals, possibly lacking a Hilbert structure, such as the space of \emph{all} admissible strategies. 
    To leverage our results it is then necessary to find an RKHS \say{densely embedded} in the target family $\mathfrak{F}$. 
    Examples are signature kernels RKHS, as discussed in \ref{subsub:sig_properties} below.
\end{remark}

\begin{example}
    One way to construct an \emph{operator}-valued kernel $K$ is to start from a real-valued kernel $\kappa:  \Lambda^\alpha \times \Lambda^\alpha \to \bR$ and a \emph{symmetric} positive definite matrix $A: V \to V$ ($V = \bR^d$), then defining $K$ as
    $$K(\bfX, \bfY) := \kappa(\bfX, \bfY)A, \quad \forall \bfX,\bfY \in \Lambda^\alpha(V).$$
    Then if $\cH_{\kappa}$ is the RKHS of $\kappa$ one has 
    $$ \cH_{K} \xhookrightarrow{} \cH_{\kappa} \otimes V_{A},$$
    where $V_{A}$ is the Hilbert space $(V, \sprod{\cdot}{A\cdot}_V)$ and $\otimes$ is the tensor product.
    Note in fact that since $\cH_{K}$ is the closure of $\text{Span}\left\{ K(\mathbf{Z},\cdot) y ~|~  \mathbf{Z}\in \Lambda^\alpha,  y \in  V \right\}$ and since one has 
    \begin{align*}
    \sprod{K(\bfX,\cdot)v}{K(\bfY,\cdot)w}_{\cH_{K}} 
    & = \sprod{K(\bfX, \bfY)v}{w}_V 
    = \kappa(\bfX, \bfY)\sprod{Av}{w}_{V}
    \\ 
    & = \sprod{\kappa(\bfX,\cdot) \otimes v}{\kappa(\bfY,\cdot) \otimes w}_{\cH_{\kappa} \otimes V_{{A}}} 
    \end{align*}
    the inclusion is the isometric embedding extending  $K(\bfX,\cdot)v \mapsto \kappa(\bfX,\cdot) \otimes v$. An example of a real-valued kernel is the \emph{signature} kernel $K_{\text{sig}}$ which we will study in more detail in Section \ref{sect:signatures}.
\end{example}


\subsection{Feature Map $\Phi_\bfX$}
To tackle the optimisation problem (\ref{eqn:hedging}), we will be interested in evaluating the following integral
\begin{equation}\label{eqn:integral}
    \int_0^T \sprod{ f(\bfX|_{[0,t]}) }{ \D \bfX_t }_{V}, \quad f \in \cH_K.
\end{equation}
In this section we present a general technique to deal with expressions of this form, recasting (\ref{eqn:integral}) as a kernel evaluation with respect to a new kernel $\mathcal{K}_{\Phi}$. Formally the argument proceeds as follows: 
\begin{align}\label{eqn:main_obstacle_1}
    \int_0^T \sprod{ f(\bfX|_{[0,t]}) }{ \D \bfX_t }_{V} & =
    \int_0^T \sprod{ f }{ ~K(\bfX|_{[0,t]},\cdot)\D \bfX_t }_{\cH_{K}} 
    \\\label{eqn:main_obstacle_2}
    &=   \sprod{ f }{ \int_0^T K(\bfX|_{[0,t]},\cdot)\D \bfX_t }_{\cH_{K}},
\end{align}
where the first equality follows from the reproducing property (\ref{eqn:rep_property}), and the second follows from linearity of the inner product. Defining 
$\Phi : \Omega^\alpha_{T} \to \cH_{K}$ as 
$$\Phi(\bfX) \equiv \Phi_\bfX := \int_0^T K(\bfX|_{[0,t]},\cdot)\D \bfX_t \in \cH_{K}$$
we obtain the equality
\begin{equation}\label{eqn:Phi_trick}
    \int_0^T \sprod{ f(\bfX|_{[0,t]}) }{ \D \bfX_t }_{V} = \sprod{f}{\Phi_\bfX}_{\cH_{K}}.
\end{equation}
We can think of $\Phi$ as a feature map with corresponding \emph{real-valued} kernel
\begin{equation}
    \begin{aligned}
         & \mathcal{K}_{\Phi} : \Omega^\alpha_T \times \Omega^\alpha_T \to \bR 
         \\
         & \mathcal{K}_{\Phi}(\bfX,\bfY) = \sprod{\Phi_\bfX}{\Phi_{\bfY}}_{\cH_{K}}
    \end{aligned}
\end{equation}
and associated RKHS $\cH_{\Phi}$. 


\begin{remark}
    Since $\cH_{\Phi}$ is the completion of $Span\{\Phi_\bfX : \bfX \in \Omega^\alpha_T\} \subseteq \cH_{K}$ under the scalar product $\sprod{\cdot}{\cdot}_{\cH_{K}}$ one has the natural embedding
    $\cH_{\Phi} \xhookrightarrow{i} \cH_{K}$, and through this $\cH_{K} = \cH_{\Phi} \oplus \cH_{\Phi}^{\perp}$.
    Moreover, the evaluations
    \[
    ev_\bfX^{K}(f) = \{ v \mapsto \sprod{f}{K(\bfX,\cdot)v}_{\cH_{K}}\}\in V^* \simeq V 
    \text{ and }
    ev_\bfX^{\Phi}(f) = \sprod{f}{\Phi_\bfX}_{\cH_{K}} \in \bR
    \]
    on these two spaces induce two different inclusions
    \[
    \cH_{K} \xhookrightarrow{\iota_{K}} \{ {\Lambda^\alpha} \to V \} \quad \text{and} \quad
    \cH_{\Phi} \xhookrightarrow{\iota_{\Phi}} \{ {\Omega^\alpha_T} \to \bR \}.
    \]
\end{remark}

\subsection{Recasting the Problem}
Given (\ref{eqn:Phi_trick}) the quadratic hedging problem (\ref{eqn:hedging}) becomes a traditional quadratic kernel minimization of form:
\begin{equation}
    \inf_{f \in {\cH_{K}}} \E^x_\bQ\left[\left( \pi(\bfX|_{[0,T]})  - \pi_0 -  \sprod{f}{\Phi_\bfX}_{\cH_{K}} \right)^2 \right].
\end{equation}

In practice the minimizer is approximated by replacing $\bQ$ with an empirical measure. In this finite-data setting, a regularization term $\epsilon \norm{f}^2_{\cH_{K}}$ is commonly added to ensure well-posedness of the problem with the assurance that as $\epsilon \to 0$ the regularized solution converges to a (possibly non-unique) un-regularised minimizer.
With the addition of such a regularisation term, the problem reads 
\begin{equation}\label{eqn:hedging_Phi}
    \inf_{f \in {\cH_{\Phi}}} \E^x_\bQ\left[\left( \pi(\bfX|_{[0,T]})  - \pi_0 -  \sprod{f}{\Phi_\bfX}_{\cH_{\Phi}} \right)^2 \right] \left( + \epsilon \norm{f}^2_{\cH_{\Phi}} \right),
\end{equation}
since one has $\cH_{K} = \cH_{\Phi} \oplus \cH_{\Phi}^{\perp}$ and $f$ enters only trough the dot-product with elements of $\cH_{\Phi}$. As a corollary of \cite[Theorem 2]{regKernels} we have the following result:
\begin{theorem}\label{thm:repr_theorem_body}
    Assume $\cX \subset \Omega^\alpha_T$ to be locally compact and second countable such that $\mathcal{K}_{\Phi}: \cX \times \cX \to \bR$ is measurable. Let  $\mathbb{Q}$ be a measure supported on $\cX$. If
    \[
    \int_{\cX} \mathcal{K}_{\Phi}(\bfX, \bfX) d\bQ(X) < +\infty
    \quad
    \int_{\cX} \pi(\bfX)^2 d\bQ(\bfX) < +\infty
    \]
    then for any $\lambda > 0$ one has 
    \begin{equation*}
        \begin{aligned}
            f^* \in & \arg\min_{f \in \cH_{\Phi}} 
        \left\{
         \E^x_\bQ\left[\left( \pi(\bfX|_{[0,T]})  - \pi_0 - \int_0^T \left\langle f(\bfX|_{[0,t]}), \D \bfX_t \right\rangle_{\bR^d}\right)^2 \right] 
        + \frac{\lambda}{2} \norm{f}_{\cH_{\Phi}}
        \right\}
        \end{aligned}
    \end{equation*}
    if and only if there is $\alpha^* \in L^2_{\bQ}(\cX)$ such that 
   \begin{equation}\label{eqn:quad_hedge_alpha}
        \left( \frac{1}{2}Id + \frac{1}{\lambda} \Xi_{\bQ}^{\Phi} \right) (\alpha^*) (\bfX) = \pi_0 - \pi(\bfX),
    \end{equation}
    where $\Xi_{\bQ}^{\Phi}: L^2_{\bQ} \to L^2_{\bQ}$ is defined as
    \[
    \Xi_{\bQ}^{\Phi}(\alpha)(\bfX) = \E_{\bfY \sim \bQ} [ \alpha(\bfY)\mathcal{K}_{\Phi}(\bfY,\bfX) ]
    = 
    \sprod{\E_{\bfY \sim \bQ} [ \alpha(\bfY)\Phi_\bfY]}{\Phi_\bfX}_{\cH_{K}}.
    \]
    If such an $\alpha^*$ exists moreover
    \begin{equation}
        f^* =  - \frac{1}{\lambda} \E_{\bQ}[\alpha^*(\bfX) \Phi_{\bfX}] \in \cH_{\Phi} \subseteq \cH_{K}
    \end{equation}
\end{theorem}

\begin{proof}
    This result is a simple corollary of \cite[Theorem 2]{regKernels} with 
    $Y = \bR$, 
    $$d\rho(\bfX, y) = d\bQ(\bfX) \otimes \delta^y_{\pi(\bfX) - \pi_0}, \quad V(x, y) := (y - x)^2, $$
    $\mathcal{B} = \{0\}$ and $f \in \cH_{\Phi}$ so that 
    \(
    ev_{\bfX}^\Phi (f) = \sprod{f}{\Phi_\bfX}_{\cH_\Phi}
    \). The theorem then states that a minimizer $f^*$ exists if and only if there is $\alpha^* \in L^2_\bQ(\mathcal{X})$ such that
    \[
    \alpha^*(\bfX) = 2 \left(\frac{1}{2(\frac{\lambda}{2})} \E_{\bfY \sim \bQ} [ \alpha^*(\bfY)\mathcal{K}_{\Phi}(\bfY,\bfX) ]  - \pi(\bfX) + \pi_0 \right)
    \]
    \emph{i.e.} if and only if
    \[
    \frac{1}{2}\alpha^*(\bfX) = \frac{1}{\lambda}\Xi_{\bQ}^{\Phi}(\alpha^*)(\bfX) + \pi_0 - \pi(\bfX).
    \]
\end{proof}

\begin{remark}
Theorem \ref{thm:repr_theorem_body} poses the stringent condition of \emph{local compactness} on $\cX$, making the result not directly applicable even to the case of Brownian paths. Practically though, as mentioned before, $\mathbb{Q}$ is taken as an empirical measure of finitely many market trajectories, acting as collocation points, which is obviously compactly supported. Alternatively, the problem can be alleviated by, as it is customary in the literature, looking for approximate solutions on compact sets having almost full measure (see Section \ref{subsub:sig_properties}).
\end{remark}

\begin{corollary}
    For any path $\bfX \in \Lambda^\alpha(\bR^d)$, one has the following representation for the portfolio weights $f^*(\bfX|_{[0,s]})\in \bR^d$ at time $s \in [0,T]$:
    \begin{equation}
        \begin{aligned}
            [f^*(\bfX|_{[0,s]})]_i &= \sprod{f^*(\bfX|_{[0,s]})}{e_i}_{\bR^d}
            = \sprod{f^*}{K(\bfX|_{[0,s]},\cdot)e_i}_{\cH_{K}}
            \\
            & = - \frac{1}{\lambda} \E_{\bfY \sim \bQ}[\alpha^*(\bfY) \sprod{\Phi_{\bfY}}{K(\bfX|_{[0,s]},\cdot)e_i}_{\cH_{K}}]
            \\
            & = - \frac{1}{\lambda} \E_{\bfY \sim \bQ}[\alpha^*(\bfY) 
            \int_{t=0}^T \sprod{K(\bfY|_{[0,t]},\cdot)\D\bfY_t}{K(\bfX|_{[0,s]},\cdot)e_i}_{\cH_{K}}]
            \\
            & = - \frac{1}{\lambda} \E_{\bfY \sim \bQ}[\alpha^*(\bfY) 
            \int_{t=0}^T \sprod{K(\bfX|_{[0,s]}, \bfY|_{[0,t]})e_i}{\D \bfY_s}_{\bR^d}]
        \end{aligned}
    \end{equation}
\end{corollary}

\begin{remark}
    When the measure $\bQ$ is the empirical measure of a set $\bX = \{\bfX_1, \dots, \bfX_n\}$ then $L^2_{\bQ}(\cX)$ is identified with $\bR^n$ via $\delta_{\bfX_i} \mapsto e_i$.
    Then $\Xi^{\Phi}_{\bQ}$ is identified with the Gram matrix of $\mathcal{K}_{\Phi}$, rescaled by $\frac{1}{n}$. Defining the map $\Gamma_{\bQ} : \Lambda^\alpha \to \bR^{d \times N}$ by 
    \[
    [ \Gamma_{\bQ}(\bfY)]_{i,j} := \int_{t=0}^T \sprod{K(\bfY, \bfX_j|_{[0,t]})e_i}{\D(\bfX_j)_t}_{\bR^d}
    = [ ev^{K}_{\bfY}(\Phi_{\bfX_j}) ]_i
    \]
    then the portfolio has the explicit weights
    \begin{equation}\label{eqn:gram_fit}
    \bR^d \ni  f^*(\bfY|_{[0,s]}) =  \frac{1}{n \lambda}\Gamma_{\bQ}(\bfY|_{[0,s]})\left( \frac{1}{2}Id_n + \frac{1}{n\lambda} \mathcal{K}_{\Phi}(\bX,\bX) \right)^{\dag} [\pi(\bX) - \pi_0\boldsymbol{1}_n]
    \end{equation}
\end{remark}

\subsection{Well Posedness}



To make the above arguments rigorous, we need to justify that the integrals in (\ref{eqn:main_obstacle_1}) and (\ref{eqn:main_obstacle_2}) are well-defined. The novelty of the following result is the proof that if the kernel functional is controlled by $\bfX$, then any element of the corresponding RKHS is also controlled as-well.

\begin{proposition}[Rough integral]\label{prop:Phi_rough_all}
    Let $\alpha \in (0,1)$ and $N = \lfloor{\frac{1}{\alpha}}\rfloor$.
    Assume $\bfX \in \mathscr{C}^{\alpha}([0,T];\bR^d)$ is an $\alpha$-H\"older rough path 
    $\mathbf{X} = (1, X^1,\cdots,X^N): \Delta_{0,T} \to T^{(N)}(\bR^d)$.
    Assume moreover the existence of a path 
    \[
    \boldsymbol{K}(X) = (K^0(X)_s, \cdots, K^{N-1}(X)_s) : [0,T] \to \bigoplus\limits_{k=0}^{N-1} \mathcal{L}((\bR^d)^{\otimes k}; \mathcal{L}(\bR^d; \cH_{K})) 
    \]
    {lifting} $K(\bfX|_{[0,s]}, \cdot)$ (\emph{i.e.} $K^0(X)_s = K(X|_{[0,s]}, \cdot)$)
    and such that, for every word $I$ with $|I|\leq N - 1$ one has
    \begin{equation}
        \norm{\boldsymbol{K}(X)_t[e_I] - \boldsymbol{K}(X)_s[\mathbf{X}_{s,t} \otimes e_I]}_{\mathcal{L}(\bR^d; \cH_{K})} \leq C |t-s|^{(N-|I|)\frac{1}{p}}
    \end{equation}
    \begin{equation}
        \norm{\boldsymbol{K}(X)_t[e_I] - \boldsymbol{K}(X)_s[e_I]}_{\mathcal{L}(\bR^d; \cH_{K})} \leq C |t-s|^{\frac{1}{p}}.
    \end{equation}
    Then the integral 
    \[
        \Phi_X := \int_0^T K(\bfX|_{[0,t]}, \cdot) \D \mathbf{X}_t := \lim\limits_{|\mathcal{P}| \to 0} \sum\limits_{[s,t]\in \mathcal{P}}
        \sum_{|I| \leq N-1}\sum_{k =1}^d  \mathbf{X}^{Ik}_{s,t}(\boldsymbol{K}(X)_s[e_{I}])[e_k] \in \cH_{K}
    \]
    is well-defined as a rough integral. Furthermore, for any $f \in \cH_{K}$, $t \mapsto f(\bfX|_{[0,t]})$ is itself integrable against $\D \bfX$ in a rough sense and one has
    \begin{equation}\label{eqn:core_swap_rough}
        \int_0^T \left\langle f(\bfX|_{[0,t]}), \D \mathbf{X}_t \right\rangle_{\bR^d} = \left\langle f, \Phi_\bfX \right\rangle_{\cH_{K}}.
    \end{equation}
\end{proposition}

\begin{proof}
    See Appendix \ref{app:sect:feature_map}
\end{proof}


\subsubsection{The Kernel $\mathcal{K}_{\Phi}$}

We have remarked above how in the discrete case in which $\bQ$ is the empirical measure of a set $\bX$ the Gram matrix $\mathcal{K}_\Phi(\bX, \bX)$ plays a crucial role in the computation of the optimal strategy. Here we show how, given two rough paths $\bfX, \bfY$, the value of $\mathcal{K}_{\Phi}(\bfX, \bfY)$ can be computed as a double rough integral:

\begin{proposition}\label{prop:double_integral}
    Let $\bfX, \bfY \in \mathscr{C}^{\alpha}([0,T]; \bR^d)$ and let the operator-valued kernel $K$ admit lifts against both $\bfX$ and $\bfY$ in the sense of Proposition \ref{prop:Phi_rough_all}. 
    Then
    \begin{align*}
        \mathcal{K}_{\Phi}(\bfX,\bfY) & = 
        \int_{s=0}^T 
        \sprod{\int_{t=0}^T K(\bfY|_{[0,t]}, \bfX|_{[0,s]}) d\mathbf{Y}_t}{\D \mathbf{\bfX}_s}_{V}
        \\
        & = 
        \int_{t=0}^T 
        \sprod{\int_{s=0}^T K(\bfX|_{[0,s]}, \bfY|_{[0,t]}) \D\mathbf{X}_s}{\D\mathbf{Y}_t}_{V} \in \bR.
    \end{align*}
\end{proposition}



\begin{proof}
    First note that the last equality follows from the properties of $\Phi_{\cdot}$, in fact
    \[
        \sprod{\int_{s=0}^T K(\bfX|_{[0,s]}, \bfY|_{[0,t]}) \D \mathbf{X}_s}{z}_{V}
        =
        \sprod{ev^{K}_{\bfY|_{[0,t]}}[ \Phi_\bfX ] }{z}_{V},
    \]
    hence
    \begin{align*}
        \int_{t=0}^T 
        \sprod{\int_{s=0}^T K(\bfX|_{[0,s]}, \bfY|_{[0,t]}) \D\mathbf{X}_s}{\D\mathbf{Y}_t}_{V}
         = &
         \int_{t=0}^T 
        \sprod{ev^{K}_{\bfY|_{[0,t]}[\Phi_\bfX]}}{\D\mathbf{Y}_t}_{V}
        \\
        = & 
        \sprod{\Phi_\bfX}{\Phi_\bfY} = \mathcal{K}_{\Phi}(\bfX,\bfY).
    \end{align*}

    Since the same argument shows 
    \[
    \int_{s=0}^T 
        \sprod{\int_{t=0}^T K(\bfY|_{[0,t]}, \bfX|_{[0,s]}) d\mathbf{\bfY}_t}{\D \mathbf{X}_s}_{V} = \sprod{\Phi_\bfX}{\Phi_\bfY}
    \]
    the equality holds. 
\end{proof}

\section{Signature Kernels for Hedging}
\label{sect:signatures}

In this section we present a family of Kernels, the \emph{signature kernels} satisfying the above conditions, so that the constructions above are well-defined.  
Recall:

\begin{proposition}[Thm. 3.7 \cite{lyons2007differential}]
Given $\mathbf{X} \in \mathscr{C}^{\alpha}([0,1];\bR^d)$ there exists a unique element $Sig(\mathbf{X})$ of $\Delta \to T((\bR^d))$ extending $\mathbf{X}: \Delta \to T^{\lfloor{\frac{1}{\alpha} \rfloor}}(\bR^d)$, satisfying Chen's relation
\[
    Sig(\mathbf{X})_{s,u} \otimes Sig(\mathbf{X})_{u,t} = Sig(\mathbf{X})_{s,t}
\]
and such that 
\[
\forall I, ~ \exists C_{|I|} \geq 0, ~ \forall s \leq t, ~ |\sprod{Sig(\mathbf{X})_{s,t}}{e_I}| \leq C_{|I|}|t-s|^{|I|\alpha}.
\]
\end{proposition}

Given any $X \in \mathcal{C}^{\alpha}([0,1];\bR^d)$ we denote by $\hat X \in \mathcal{C}^{\alpha}([0,1];\bR^{d+1})$ the \emph{time-augmented} path $\hat X_t := (t, X_t)$.
Then if $X$ can be lifted to a rough path $\mathbf{X} \in \mathscr{C}^{\alpha}([0,1];\bR^d)$, $\hat X$ can be lifted as-well.
 Over \emph{geometric} and \emph{time-augmented} rough paths the signature becomes \emph{injective} (\emph{c.f.} corollary of \cite[Thm. 1.1.]{boedihardjo2015signatureroughpathuniqueness}). 
 As mentioned in the introduction, since signatures of geometric rough paths $\mathbf{X} \in \mathscr{C}^{0, \alpha}_g([0,1]; \bR^d)$ satisfy the \emph{shuffle identity} 
 \[
 \forall I,J, ~  \sprod{Sig(\mathbf{X})}{e_I}\sprod{Sig(\mathbf{X})}{e_J} = \sprod{Sig(\mathbf{X})}{e_I \shuffle e_J},
 \]
 a Stone-Weiestrass argument shows that the signature becomes a \emph{universal} feature map, meaning that the corresponding RKHS is dense in continuous functions over compact subsets of $\mathscr{C}^{0, \alpha}_g([0,1]; \bR^d)$.

\subsection{Lead-Lag and Signature Kernel Hedging}
\label{subsub:sig_properties}

The observations above are particularly useful in the case of semi-martingales: in fact (\emph{c.f.} \cite[Ch. 14]{friz2010multidimensional}) any continuous semi-martingale $X$ can be a.s. canonically lifted to geometric $\frac{1}{2^+}$-rough paths, the lift corresponding to the level-2 enhancement
\[
\mathbf{X}^{\text{Strat}}_{s,t} = (X_t - X_s, \int_s^t (X_r - X_s) \otimes \circ \D X_r).
\]
Thus, any continuous map on a compact subset of sample paths of $\mathbf{X}^{\text{Strat}}_{s,t}$ can be approximated a.s. to arbitrary precision by a linear map on their time-augmented Stratonovich lift $Sig^{\text{Strat}}(\hat{X}) := Sig(\hat{\mathbf{X}}^{\text{Strat}})$.
To do this lift we can associate the kernel
\begin{align}
    K_{sig}^{\text{Strat}}(X|_{0,s},Y|_{0,t}) &= \sprod{Sig^{\text{Strat}}(\hat X)_{0,s    }}{~Sig^{\text{Strat}}(\hat Y)_{0,t}}_{T((V))}.
\end{align}

Continuous semi-martingales can also be lifted to geometric $\frac{1}{2^+}$-rough paths in an It\^o sense as well via the level-2 integral
\[
\mathbf{X}^{\text{It\^o}}_{s,t} = (X_t - X_s, \int_s^t (X_r - X_s) \otimes \D X_r).
\]

Note that the level-2 components of these lifts are related by
\[
\int_s^t (X_r - X_s) \otimes \D X_r = \int_s^t (X_r - X_s) \otimes \circ \D X_r - \frac{1}{2} \int_s^t \D [X, X]_{r},
\]
so that the information about the two lifts can be \say{packaged} together, in fact it can be used to construct a single geometric rough-path (\emph{c.f.} \cite[2.15]{lyons2019nonparametricpricinghedgingexotic})

\begin{definition}(Lead-Lag lift)
    Given a continuous semi-martingale $X: [0,1] \to \R^d$, we define its Lead-Lag lift as the $T^{(2)}(\bR^{2d})$-valued \emph{geometric} $\frac{1}{2^+}$-H\"older rough path
    \begin{equation}
        \mathbf{X}^{LL} := \left(1, \begin{pmatrix} ~ X ~ \\ ~ X ~ \end{pmatrix} , \begin{pmatrix}
            \bX & \bX - \frac{1}{2}[X,X] ~
            \\
            ~\bX  + \frac{1}{2}[X,X] & \bX
        \end{pmatrix} \right),
    \end{equation}
    where 
    \[
    \bX_{s,t} = \int_s^t (X_r - X_s) \otimes \circ \D X_r
    \]
\end{definition}

This leads to the following \say{kernelization} of \citep[Thm. 4.7]{lyons2019nonparametricpricinghedgingexotic}:

\begin{theorem}\label{thm:Immanol}
    Let X be a continuous 
    {
    $\bR^d$-valued
    }
    semi-martingale, $\bQ$ the measure induced by its Lead-Lag lift, and 
    $$
    \mathfrak{F} := \left\{f \in 
    {
    C^0(\mathscr{C}^{\frac{1}{2^+}}_g([0,T]; \bR^{d+1}); \bR^d)
    }~|~ \E_{\mathbb{Q}}\left[ \left(\int_0^T \sprod{f(\hat{\bfX}^{\text{Strat}}|_{[0,t]})}{\D \bfX^{\text{It\^o}}_t} \right)^2\right] < \infty \right\}
    $$
    Let $f^* \in \mathfrak{F}$ be the minimizer of 
    \[
    a := 
    \inf_{f \in \mathfrak{F}} \E_\bQ\left[\left( \pi( \hat{\bfX}^{LL})  - \pi_0 - \int_0^T \left\langle f(\hat{\bfX}^{\text{Strat}}|_{[0,t]}), \D \bfX^{\text{It\^o}}_t \right\rangle\right)^2 \right].
    \]
    Given any $\epsilon > 0$ there exists
    \begin{enumerate}[label=\roman*.]
        \item  a \emph{compact} set $\mathcal{X}_{\epsilon} \subseteq  \mathcal{C}^{\frac{1}{2^+}}([0,T]; \bR^{d})$,
        \item a strategy $f_\epsilon \in \cH_{K_{sig}^{\text{Strat}}}$
    \end{enumerate}
    such that:
    \begin{enumerate}[label=\roman*.]
        \item $\mathbb{Q}[\mathcal{X}_\epsilon] > 1 - \epsilon$,
        \item $|f^*(\hat{\bfX}^{\text{Strat}}|_{[0,t]}) - f_\epsilon({X}|_{[0,t]})| < \epsilon ~ \forall X \in \mathcal{X}_\epsilon$ and $t \in [0, T]$ 
        \item $|a - a_\epsilon| < \epsilon$, where
        \[
        a_\epsilon := \E_\bQ\left[\left( \pi(\hat{\bfX}^{LL})  - \pi_0 - \int_0^T \left\langle f_\epsilon(X|_{[0,t]}), \D \bfX^{\text{It\^o}}_t \right\rangle\right)^2 ~;~ \mathcal{X}_\epsilon ~\right].
        \]
    \end{enumerate}
\end{theorem}

{
\begin{proof}
    We just reframed \citep[Thm. 4.7]{lyons2019nonparametricpricinghedgingexotic} in the language of RKHSs, leveraging the fact that their \emph{signature-payoffs} of type $X|_{[0,t]} \mapsto \sprod{l}{\hat{\bfX}^{\text{Strat}}|_{[0,t]}}_{T((\bR^{d+1}))}$ are elements of the space $\cH_{K_{sig}^{\text{Strat}}}$.
\end{proof}
}

Since both $\mathbf{X}^{\text{It\^o}}$ and $\mathbf{X}^{\text{Strat}}$ are projections of $\mathbf{X}^{LL}$, 
the minimization 
\[
\min_{f \in \cH_{K_{sig}^{\text{Strat}}}}\E_\bQ\left[\left( \pi(\hat{\bfX}^{LL})  - \pi_0 - \int_0^T \left\langle f(X|_{[0,t]}), \D \bfX^{\text{It\^o}}_t \right\rangle\right)^2 ~;~ \mathcal{X}_\epsilon ~\right].
\]
falls in the kernel framework presented above. 
In Appendix \ref{app:wellDefSig} we show that $K_{sig}^{\text{Strat}}$ can be integrated against $\bfX^{\text{It\^o}}$, so that Prop. \ref{prop:Phi_rough_all} can be applied and the problem recast as 
\[
\min_{f \in \cH_{K_{sig}^{\text{Strat}}}}
\E_\bQ\left[\left( \pi(\hat{\bfX}^{LL})  - \pi_0 - \left\langle f, \Phi_{\hat{\bfX}^{LL}} \right\rangle_{\cH_{K_{sig}^{\text{Strat}}}} \right)^2 ~;~ \mathcal{X}_\epsilon ~\right].
\]
for the well defined vector
\[
\Phi_{\hat{\bfX}^{LL}} := \int_0^T K_{sig}^{\text{Strat}}(X|_{[0,t]}, \cdot) \D \bfX^{\text{It\^o}}_t \in \cH_{K_{sig}^{\text{Strat}}}
\]

\begin{remark}
In practice, the Gram matrix with entries 
\begin{equation}\label{eqn:gram_sig}
    \int_0^T \int_0^T  K_{sig}^{\text{Strat}}(X|_{[0,t]}, Y|_{[0,s]}) \D \bfX^{\text{It\^o}}_t \D \bfY^{\text{It\^o}}_s 
\end{equation} 
has to be computed having only access to $M$ samples 
$$X, Y : \mathcal{T}_M := \{0 = t_0 < \cdots < t_M = T\} \to \bR^d.$$
This can be achieved by first computing with high accuracy $K_{sig}^{\text{Strat}}(X|_{[0,t]}, Y|_{[0,s]})$ on the grid $\mathcal{T}_M \times \mathcal{T}_M$ leveraging the geometric nature of Stratonovich lifts, for example by solving the signature kernel PDE (\cite{salvi2021signature}) for smooth interpolations of the samples, and then approximating the double integral with a standard Euler scheme.
\end{remark}

\begin{remark}
    Similarly, one can consider the above expression entirely formulated in terms of It\^o signature kernels, reducing the problem to 
    \[
    \min_{f \in \cH_{K_{sig}^{\text{It\^o}}}}
    \E_\bQ\left[\left( \pi(\hat{\bfX}^{\text{It\^o}})  - \pi_0 - \left\langle f, \Phi_{\hat{\bfX}^{\text{It\^o}}} \right\rangle_{\cH_{K_{sig}^{\text{It\^o}}}} \right)^2 ~;~ \mathcal{X}_\epsilon ~\right].
    \]
    Given a continuous semi-martingale $X : [0,1] \to \R^d$, it is easy to check that the following expression holds
    \begin{equation}
        Sig(\mathbf{X}^{\text{It\^o}})_{s,t} = 1 + \int_s^t Sig(\mathbf{X}^{\text{It\^o}})_{s,r} \otimes (\circ dX_r - \frac{1}{2}d[X,X]_r),
    \end{equation}
    where $[X,X]$, a.s. in $C^{1-var}([0,1];(\R^d)^{\otimes 2})$, is the quadratic variation process of $X$. As a corollary, we can derive the following result, which is of independent interest to the rest of the paper (see Appendix \ref{app:sect:ItoSigKer} for a proof).
    \begin{theorem}
        Given continuous semimartingales $X, Y : [0,1] \to \R^d$ their It\^o Signature Kernel can be written only in terms of Stratonovich integrals as the solution of the following system of 3 SDEs:
        
        \begin{equation*}
            \R^d \ni G(s,t) = \int_{\tau = 0}^t \int_{\sigma = 0}^s K^{It\hat{o}}(X,Y)_{\sigma,\tau} d[Y,Y]^T_{\tau} \circ dX_{\sigma} - \frac{1}{2} \int_{\tau = 0}^t 
            d[Y, Y]^T_{\tau} F(s,\tau),
        \end{equation*}
        \begin{equation*}
            \R^d \ni F(s,t) = \int_{\tau = 0}^t \int_{\sigma = 0}^s K^{It\hat{o}}(X,Y)_{\sigma,\tau} d[X,X]^T_{\sigma} \circ dY_{\tau} - \frac{1}{2} \int_{\sigma = 0}^s
            d[X, X]^T_{\sigma} G(\sigma, t),
        \end{equation*}
        \begin{align*}
           \R \ni K^{It\hat{o}}(X,Y)_{s,t} = 
            \hspace{5pt} 1 
            &+  \int_0^s \int_0^t 
            K^{It\hat{o}}(X,Y)_{\sigma,\tau}  
            \sprod{\circ dX_{\sigma}}{\circ dY_{\tau}}_{\R^d}
            \\
            &+ \frac{1}{4} \int_0^s \int_0^t 
            K^{It\hat{o}}(X,Y)_{\sigma,\tau} Tr(d[X,X]^T_{\sigma} d[Y,Y]_{\tau})
            \\
            & - \frac{1}{2} 
            \int_{0}^s \sprod{G(\sigma,t)}{\circ dX_{\sigma}}_{\R^d}
            - \frac {1}{2} 
            \int_{0}^t \sprod{F(s,\tau)}{\circ dY_{\tau}}_{\R^d}.
        \end{align*}
        
        \end{theorem}
        
\end{remark}


This result could help address a practical challenge: Stratonovich signature kernels can be computed with high accuracy by leveraging smooth approximations of the driving paths and solving the associated PDEs. However, the same does not hold true for their It\^o counterparts, owing to the non-geometric nature of this lift.

By expressing the Itô kernel in terms of Stratonovich integrals, one can apply these smoothing techniques. Naturally, this comes with a trade-off, as the quadratic variation must now be provided as input to the method.
A comprehensive exploration of these numerical aspects is left for future research.




\section{A General Theory}
\label{sec:general_theo}

From the previous motivating example of quadratic Hedging it seems that this \say{kernel trick} is suitable only for quadratic problems where only information from the underlying price processes is considered.

Fortunately the method generalizes with little effort to more general settings as we are about to show in the present section. 

\subsection{General Setting}

Fix a probability space $(\cX, \mathcal{F}, \mu)$, this will be our \emph{base space}. Think of this as the space of all the available data, which should include the price processes of the assets one wants to trade but could moreover contain e.g. sentiment information, market regime indicators, aggregate market data, price processes of other assets.
All these datapoints will have to be \say{synthesized} together in a suitable operator valued Kernel $K$ and the tradable asset prices will have to be singled out trough some signal map $\xi$.
Fix the following 3 \say{items}:

\begin{enumerate}[label=\roman*.]
    \item A TVS feature space $\mathcal{G}_{\psi}$ and a \emph{feature extractor}  $\psi : \cX \to C^0([0,T];\mathcal{G}_{\psi})$.
        
    \item A {finite dimensional} vector (\emph{market}) space $\cH_{\xi}$ and a \emph{market signal} map $\xi : \cX \to \Omega^\alpha_T(\cH_{\xi})$.

    \item An \emph{operator valued kernel} defined on $\mathcal{S} := \bigcup_{s \in [0,T]} C^0([0,s];\mathcal{G}_{\psi})$:
    $$K: \mathcal{S} \times \mathcal{S} \to \mathcal{L}(\cH_{\xi})$$
    and its corresponding Reproducing Kernel Hilbert Space 
    $$\cH_{K} \subseteq \{\mathcal{S} \to \cH_{\xi} \}.$$
\end{enumerate}

The minimisation problems we aim to solve will have the more general form
\[
\arg\min_{F \in \cH_{K}} 
        \left\{
        \int_{\cX \times Y} V(\int_0^1 \sprod{F(\psi(x)|_{[0,t]})}{\D \xi(x)_t}_{\cH_{\xi}}, y) d\tilde\mu(x,y)
        + \frac{\lambda}{2} \norm{F}_{\cH_{K}}
        \right\},
\]
where $\tilde\mu$ is an extension of $\mu$ on $\cX \times Y$, with $Y$ Banach, and $V$ a $p$-Loss Function (as defined below).
As before the main ingredient in the solution will be giving a meaning and provide an understanding of the expressions
\begin{equation}
    \int_0^T \sprod{F(\psi(x)|_{[0,t]})}{\D \xi(x)_t}_{\cH_{\xi}}
\end{equation}
for $x \in supp(\mu) \subseteq \cX$ and $F \in \cH_{K}$. 

\begin{remark}
    The quadratic hedging setting considered above can readily be recovered with the following choices: $\cX = \Omega^\alpha_T(V)$, $\mu = \mathbb{Q}^x$, $Y = \bR$, $d\tilde\mu(x,y) := d \mu(x) \otimes d\delta_{\pi(x)}(y)$ with $V(x,y) = (y - \pi_0 - x)^2$ and $\mathcal{G}_{\psi} = \cH_{\xi} = V$, $\psi = \xi = Id_{\Omega^\alpha_T}$.
\end{remark}

\subsection{Feature Map $\Phi_x$}

The arguments and the result are basically the same, in fact the proof is done directly in this general case in the Appendix, only the language changes, becoming richer to accomodate the new objects.

\begin{theorem}\label{thm:generic_Phi}
    Let $\alpha \in (0, 1)$ and $N = \lfloor{\frac{1}{\alpha}}\rfloor$ and $x \in \cX$.
    Let $\xi(x) \in \Omega_T^\alpha(\cH_\xi)$ be the rough path
    $$\boldsymbol{\xi}(x) = (1, {\xi}^1(x),\cdots,\xi^N(x)): \Delta_{0,T} \to T^{(N)}(\cH_{\xi})$$
    and assume the existence of a path 
    \[
    \boldsymbol{K}(x) = (K^0(x), \cdots, K^{N-1}(x)) : [0,T] \to \bigoplus\limits_{m=0}^{N-1} \mathcal{L}(\cH_{\xi}^{\otimes m}; \mathcal{L}(\cH_{\xi}; \cH_{K})) 
    \]
    extending $K^0(x)_s := K(\psi(x)|_{[0,s]}, \cdot)$ and such that for every $m \leq N - 1$
    \begin{equation}
        \norm{K^m(x)_t - \boldsymbol{K}(x)_s[\boldsymbol{\xi}(x)_{s,t} \otimes (\cdot)]}_{\mathcal{L}(\cH_{\xi}^{\otimes m}; \mathcal{L}(\cH_{\xi}; \cH_{K})) } \leq C |t-s|^{(N-m)\frac{1}{p}},
    \end{equation}
    \begin{equation}
        \norm{K^m(x)_t - K^m(x)_s}_{\mathcal{L}(\cH_{\xi}^{\otimes m}; \mathcal{L}(\cH_{\xi}; \cH_{K})) } \leq C |t-s|^{\frac{1}{p}}.
    \end{equation}
    then it is well defined in a rough sense\footnote{we implicity use the inclusion $\mathcal{L}(\cH_{\xi}^{\otimes m}; \mathcal{L}(\cH_{\xi}; \cH_{K}))  \xhookrightarrow{} \mathcal{L}(\cH_{\xi}^{\otimes m +1}; \cH_{K})$} 
    \begin{align*}
        \Phi_x & := \int_0^T K(\psi(x)|_{[0,t]}, \cdot) \D \boldsymbol{\xi}(x)_t 
        \\
        & := \lim\limits_{|\mathcal{P}| \to 0} \sum\limits_{[s,t]\in \mathcal{P}}
        \sum_{m = 0}^{N-1}  K^m(x)_s [\boldsymbol{\xi}(x)^{m+1}_{s,t}]
        \in \cH_{K}.
    \end{align*}
    Furthermore each $t \mapsto f(\psi(x)|_{[0,t]})$ is itself integrable against $\D \boldsymbol{\xi}$ in a rough sense and
    \begin{equation}\label{eqn:core_swap_rough}
        \int_0^T \left\langle f(\psi(x)|_{[0,t]}), \D \boldsymbol{\xi}(x)_t \right\rangle_{\cH_{\xi}} = \left\langle f, \Phi_x \right\rangle_{\cH_{K}}.
    \end{equation}
    
\end{theorem}
\begin{proof}
    See Appendix \ref{app:sect:feature_map}
\end{proof}

As before $\Phi: \cX \to \cH_{K}$ defines a new real valued kernel on $\cX$:
\begin{equation}
    \begin{aligned}
         & \mathcal{K}_{\Phi} : \cX \times \cX \to \bR 
         \\
         & \mathcal{K}_{\Phi}(x,\tilde x) = \sprod{\Phi_x}{\Phi_{\tilde x}}_{\cH_{K}}.
    \end{aligned}
\end{equation}
and the associated RKHS $(\cH_{\Phi}, \sprod{\cdot}{\cdot}_{\cH_{\Phi}})$. 

\begin{lemma}
    In the hypotheses of the previous result
    \begin{align*}
        \mathcal{K}_{\Phi}(x,\tilde x) & = 
        \int_{s=0}^T 
        \sprod{\int_{t=0}^T K(\psi(\tilde x)|_{[0,t]}, \psi(x)|_{[0,s]}) \D \boldsymbol{\xi}(\tilde x)_t}{\D \boldsymbol{\xi}(x)_s}_{\cH_{\xi}}
        \\
        & = 
        \int_{t=0}^T 
        \sprod{\int_{s=0}^T K(\psi(x)|_{[0,s]}, \psi(\tilde x)|_{[0,t]}) \D \boldsymbol{\xi}(x)_s}{\D \boldsymbol{\xi}(\tilde x)_t}_{\cH_{\xi}} \in \bR
    \end{align*}
\end{lemma}
\begin{proof}
    Just a repetition of the arguments of Proposition \ref{prop:double_integral}.
\end{proof}

Since $\cH_{\Phi}$ is the completion of $Span\{\Phi_x : x \in \cX \} \subseteq \cH_{K}$ under $\sprod{\cdot}{\cdot}_{\cH_{K}}$ one has the natural inclusion $\cH_{\Phi} \xhookrightarrow{i} \cH_{K}$ and trough this we also get the decomposition
$$\cH_{K} = \cH_{\Phi} \oplus \cH_{\Phi}^{\perp}.$$

Once again the evaluation functionals
    \[
    ev_Z^{K}(f) = \sprod{f}{K(Z,\cdot)[*]}_{\cH_{K}} \in (\cH_{\xi})^* \simeq \cH_{\xi}
    \text{ and }
    ev_x^{\Phi}(f) = \sprod{f}{\Phi_x}_{\cH_{K}} \in \bR
    \]
    on these two spaces are different and induce the two inclusions
    \[
    \cH_{K} \xhookrightarrow{\iota_{K}} \{ \mathcal{S} \to \cH_{\xi} \} \quad \text{and} \quad
    \cH_{\Phi} \xhookrightarrow{\iota_{\Phi}} \{ \cX \to \bR \}.
    \]

    \begin{remark}
        In this general setting the difference between the two spaces is clearer since when seen as function spaces they take different kinds of inputs.
    \end{remark}

\subsection{Noteworthy Special Cases}

\subsubsection{Induced Kernel - General}

Assume to have a well defined real valued kernel $\kappa: \cS \times \cS \to \bR$. 
One can define an operator valued kernel $K : \cS \times \cS \to \cH_{\xi}$ as $K(x,y) := \kappa(x,y)Id_{\cH_{\xi}}$.

\begin{lemma}
    Under the assumptions above $\cH_{K} \simeq \cH_{\kappa} \otimes \cH_{\xi}$.
\end{lemma} 

\begin{proof}
    Recall that $\cH_{K}$ is the closure of $\mathrm{Span}(K(x,\cdot)z ~|~ x \in \cS, z \in \cH_{\xi})$ under the scalar product 
    \begin{align*}
    \sprod{K(x,\cdot)z}{K(y,\cdot)w}_{\cH_{K}} & := \sprod{K(x,y)z}{w}_{\cH_{\xi}} 
    \\ &= \kappa(x,y)\sprod{z}{w}_{\cH_{\xi}} 
    = \sprod{\kappa(x,\cdot)}{\kappa(y,\cdot)}_{\cH_{\kappa}} \sprod{z}{w}_{\cH_{\xi}}
    \end{align*}
    so that $K(x,\cdot)z \mapsto \kappa(x,\cdot) \otimes z \in \cH_{\mathbb{K}} \otimes \cH_{\xi}$ extends to an isometry.
    
\end{proof}

Under this identification for $x \in \cX$
\[
\Phi_x = \int_0^T \kappa(\psi(x)|_{[0,t]}, \cdot) \otimes d \boldsymbol{\xi}(x)_t  \in \cH_{\kappa} \otimes \cH_{\xi}.
\]

\subsubsection{Induced Kernel - Feature Map}
Of special interest is the case where $\cH_{\kappa} \subseteq \bR^N$ and $\cH_{\xi} = \bR^{d_{\xi}}$, obtained when $\kappa(x,y) = \sprod{\mathbb{F}(x)}{\mathbb{F}(y)}_{\bR^N}$ for some feature map $\mathbb{F}: \cS \to \bR^N$.
In this setting $\cH_{\kappa} \otimes \cH_{\xi} \subseteq \bR^{N \times d_{\xi}}$ is a subset of matrices endowed with Hilbert-Schmidt norm and we can write
\begin{equation}
    \Phi_x = \int_0^T 
    \underbrace{
    \mathbb{F}(\psi(x)|_{[0,t]}) 
    }_{N \times 1}
    \underbrace{
    \D \boldsymbol{\xi}(x)_t^{\top}  
    }_{1 \times d_{\xi}}
    \in \bR^{N \times d_{\xi}}
\end{equation}
and moreover $\mathcal{K}_{\Phi}(x,y) = Tr(\Phi_x\Phi_y^{\top})$.

\vspace{10pt}
This technique is particularly relevant in the context of \emph{Neural Signature Kernels} \cite{cirone2023neural}, a family of path-space kernels that generalize signature kernels and can only be computed through feature maps.

\subsection{The Optimization Problem - Explicit Solutions}

\begin{definition}
    Given $p \in [1, +\infty)$, a function $V : \bR \times Y \to [0,+\infty)$ such that 
    \begin{enumerate}[label=\roman*.]
        \item $\forall y \in Y. \quad V(\cdot, y): \bR \to [0,+\infty)$ is convex.
        \item  V is measurable.
        \item  There are $b \geq 0$ and $a: \bR \to [0,+\infty)$ such that 
        \begin{equation}
            \begin{aligned}
                & V(x,y) \leq a(y) + b|x|^\alpha  \quad \forall (x,y) \in \bR \times Y
                \\
                & \int_{\cX \times Y} a(y) d\mu(x,y) < +\infty
            \end{aligned}
        \end{equation}
    \end{enumerate}
    is called a \emph{p-loss} function with respect to $\mu$.
\end{definition}

With the same proof of Theorem \ref{thm:repr_theorem_body}, as a corollary of \cite[Thm. 2]{regKernels} we have the following result:
\begin{theorem}\label{thm:repr_better}
    Assume the base space $\cX$ to be locally compact and second countable such that $\mathcal{K}_{\Phi}: \cX \times \cX \to \bR$ is measurable. Let $Y$ be a close subset of $\bR$ and $\tilde \mu$ be a probability measure on $\cX \times Y$. 
    If for $p \in [1,+\infty)$ it holds
    \[
    \int_{\cX \times Y} \mathcal{K}_{\Phi}(x, x)^{\frac{p}{2}} d\tilde\mu(x,y) < +\infty,
    \]
    then for any \emph{p-loss} function $V$ with respect to $\tilde \mu$ and $\lambda > 0$ one has
    \begin{equation*}
        \begin{aligned}
            F^* \in & \arg\min_{F \in \cH_{K}} 
        \left\{
        \int_{\cX \times Y} V(\int_0^1 \sprod{F(\psi(x)|_{[0,t]})}{\D \boldsymbol{\xi}(x)_t}_{\cH_{\xi}}, y) d\tilde \mu(x,y)
        + \frac{\lambda}{2} \norm{F}_{\cH_{K}}
        \right\}
        \\
        & = \arg\min_{F \in \cH_{\Phi}} 
        \left\{
        \int_{\cX \times Y} V(\sprod{F}{\Phi_{x}}_{\cH_{\Phi}}, y) d\tilde \mu(x,y)
        + \frac{\lambda}{2} \norm{F}_{\cH_{\Phi}}
        \right\}
        \end{aligned}
    \end{equation*}
    if and only if there is $\alpha^* \in L^q_{\tilde \mu}(\cX \times Y)$, with $\frac{1}{p} + \frac{1}{q} = 1$, such that 
    \begin{equation}
        \begin{aligned}
            & \alpha^*(x, y) \in \partial_1 V(ev^{\Phi}_{x}F^*, y)  \quad \tilde \mu-a.e.
            \\
            & ev^{\Phi}_{x}F^* = - \frac{1}{\lambda} \int_{\cX\times Y} \mathcal{K}_{\Phi}(x,x') \alpha^*(x') d\tilde \mu(x',y) .
        \end{aligned}
    \end{equation}
    Then for all $z\in \cH_{\xi}$ and $s \in \mathcal{S}$ one has
    \[
    \sprod{ev^{K}_{s}F^*}{z}_{\cH_{\xi}} = \sprod{F^*}{K(s, \cdot)z}_{\cH_{K}} =
    - \frac{1}{\lambda} \int_{\cX\times Y} \sprod{\Phi_{x'}}{K(s, \cdot)z}_{\cH_{K}} \alpha^*(x') d\tilde \mu(x',y)
    \]
    \emph{i.e.}
    \[
    \sprod{ev^{K}_{s}F^*}{z}_{\cH_{\xi}} =
    - \frac{1}{\lambda} \int_{\cX\times Y}  \alpha^*(x') \left( \int_0^1 \sprod{K(s, \psi(x')|_{[0,t]})z}{\D \boldsymbol{\xi}(x')_t}_{\cH_{\xi}}  \right) d\tilde \mu(x',y).
    \]
\end{theorem}

\subsection{The Optimization Problem - Functional Analysis}

By forsaking exactness in Theorem \ref{thm:repr_better} we can obtain, using basic functional analytic arguments, another representer result, of bigger scope but lesser power.

\begin{definition}
    Given a probability measure $\mu$ on the space $\cX$ we define the map $\iota_{\mu} : L_{\mu}^2(\cX;\bR) \to \cH_{\Phi} \subseteq \cH_{K}$ as 
    \[
    \iota_{\mu}\alpha = \int_{\cX} \alpha(x) \Phi_x d\mu(x)
    \]
    and the image space as $\cH_{\mu} := \iota_\mu  L_{\mu}^2(X;\bR)  \subseteq \cH_{\Phi} \subseteq \cH_{K}$. 
\end{definition}

\begin{lemma}
    If $x \mapsto \Phi_x$ is measurable and $\E_{\mu}[\mathcal{K}_{\Phi}(x,x)] < \infty$ then $\iota_{\mu}$ is a well defined bounded linear map with adjoint $\iota_{\mu}^{\top} : \cH_{K} \to L^2_{\mu}(\cX;\bR)$ given by
    \[
    \iota_{\mu}^{\top} {f}(x) := \sprod{f}{\Phi_x}_{\cH_{K}}. 
    \]
\end{lemma}

\begin{proof}

    Since for $x \in \cX$ one has
    \[
    \norm{\alpha(x)\Phi_x}_{\cH_{K}} =
    |\alpha(x)| \norm{\Phi_x}_{\cH_{K}}
    \leq |\alpha(x)|^2 + \mathcal{K}_{\Phi}(x,x)
    \]
    the map $x \mapsto \alpha(x)\Phi_x$ is measurable and absolutely integrable making $\iota_{\mu}$ a well defined linear map. By Cauchy-Schwartz we also get continuity:
    \[
    \norm{\iota_{\mu}\alpha}_{\cH_{K}} \leq 
    \int_{\cX} |\alpha(x)|\norm{\Phi_x}_{\cH_{K}} d\mu(x) \leq \norm{\alpha}_{L^2_{\mu}} 
    \E_{x \sim \mu}[\mathcal{K}_{\Phi}(x,x)]^{\frac{1}{2}}.
    \]

    Consider now the map $\iota_{\mu}^{\top}$ as above, this is linear and continuous since
    \[
        \norm{\iota_{\mu}^{\top}f}^2_{L^2_{\mu}}
        \leq 
        \norm{f}^2_{\cH_{K}} \E_{x \sim \mu}[\mathcal{K}_{\Phi}(x,x)].
    \]

   By linearity we also see that the two maps are adjoint:
    \begin{align*}
        \sprod{\iota_{\mu} \alpha}{f}_{\cH_{K}}
        &= \int  \alpha(x) \sprod{\Phi_x}{f}_ {\cH_{K}}d\mu(x)
        = \sprod{\alpha}{\sprod{f}{\Phi_{\cdot}}_{\cH_{K}}}_{L^2_{\mu}}
        = \sprod{\alpha}{\iota_{\mu}^{\top}f}_{L^2_{\mu}}.
    \end{align*}
\end{proof}

\begin{proposition}\label{prop:zero_dot}
The closure $\bar \cH_{\mu}$ of $\cH_{\mu}$ in $\cH_{\Phi}$ (or equivalently in $\cH_{K}$) is a complete subspace with the induced norm. This induces the decomposition 
$\cH_{K} = \bar \cH_{\mu} \oplus \cH_{\mu}^{\perp}$.
Moreover even if $\cH_{\Phi}^{\perp} \subseteq \cH_{\mu}^{\perp}$, whenever $f \in \cH_{\mu}^{\perp}$ then $\sprod{f}{\Phi_x}_{\cH_{K}} = 0$ $\mu$-a.s.
\end{proposition} 

\begin{proof}
    $\cH_{\mu}$ is a linear subspace since it is a linear image of a vector space. We conclude the first part since  $\bar \cH_{\mu}$ is closed.
    Now if $f \in \cH_{\mu}^{\perp}$ then for all $\alpha \in L^2_{\mu}$
    \begin{align*}
        0 = \sprod{\iota_{\mu} \alpha }{f}_{\cH_{K}} 
        &= \sprod{\alpha }{\iota_{\mu}^{\top} f}_{L^2_{\mu}} 
        = \sprod{\alpha(\cdot) }{\sprod{f}{\Phi_{\cdot}}_{\cH_{K}}}_{L^2_{\mu}}
    \end{align*}
    hence one must have $\sprod{f}{\Phi_{\cdot}}_{\cH_{K}} = 0$ in $L^2_{\mu}$
    \emph{i.e.} $\sprod{f}{\Phi_x}_{\cH_{K}} = 0$ $\mu$-a.s.
\end{proof}

\begin{remark}
Note that this does not imply that $\mu$-a.s. one has $\Phi_x \in \bar \cH_{\mu}$, to have this in general one needs some separability assumptions on $\cH_{K}$.
\end{remark}

\begin{theorem}[Representer]\label{thm:representer}
    Let $J_{\mu} : (L^2_{\mu})^{1+N} \to \bR$ be continuous and such that every input only ever appears under $\E_{\mu}$ in $J_{\mu}$. A maximizing sequence for
    \[
    \max \limits_{f \in \cH_{K}} J_\mu(\sprod{f}{\Phi_{\cdot}}_{\cH_{K}}, \{ \sprod{f}{\iota_{\mu} \beta_i(\cdot)}_{\cH_{K}} \}_{i=1
    }^N) + \epsilon \norm{f}_{\cH_{K}}
    \]
    can be found in $\cH_{\mu}$. In particular any solution $f^* \in \cH_{K}$, if any exists, of minimal norm is in $\bar \cH_{\mu}$ and can be found as the $\cH_{K}$ limit 
    \[
    f^* = \lim\limits_{n \to \infty} \iota_{\mu} \alpha_n = 
    \lim\limits_{n \to \infty} \int_{\cX} \alpha_n(x) \Phi_x d\mu(x)
    \]
    where $(\alpha_n)$  is a maximising sequence of 
    \[
    \max \limits_{\alpha \in L^2_{\mu}(\cX; \bR)} J_{\mu}(\sprod{\alpha}{\iota_{\mu}^{\top} \Phi_{\cdot}}_{L^2_{\mu}}, \{ \sprod{\alpha}{\iota_{\mu}^{\top}\iota_{\mu} \beta_i(\cdot)}_{L^2_{\mu}} \}_{i=1}^N) + \epsilon \sprod{\alpha}{\iota_{\mu}^{\top}\iota_{\mu} \alpha}_{L^2_{\mu}}.
    \]
\end{theorem}
\begin{proof}
    It follows from Proposition \ref{prop:zero_dot} that, for any $f \in \cH_{K}$, discarding the projection on $\cH_{\mu}^{\perp}$ decreases the norm but does not change the value of the quantity to maximize.
\end{proof}

\begin{remark}
    We find the quadratic hedging problem choosing
    \[
    J_{\mathbb{Q}}(Z) = - \E_{\mathbb{Q}}[(\pi(X) - \pi_0 - Z)^2] 
    \]
    but the previous result applies to richer classes of losses than Thm. \ref{thm:repr_better} as, for example, ones dealing with variance like
    \[
    J_{\mu}(Z) = \E_{\mu}[Z] - \lambda \E_{\mu}\left[(Z - \E_{\mu}[Z])^2\right]
    \]
    or exponential ones like
    \[
    J_{\mu}(Z) = \E_{\mu}\left[\exp\left\{-(\pi(X) - \pi_0 - Z)^2\right\}\right].
    \]

    On the flip side this result is a very loose existence one and does not give the means to find the maximizing sequence.
\end{remark}

\newpage

\section{Empirical Validation}
\label{sec:experiments}

    Our methodology is model-agnostic and outputs the replicating portfolio solely based on past market data. In this experimental section, we validate our theoretical framework by examining a simple case of geometric Brownian motion (GBM) under the Black-Scholes model.
 
    In this setting, the optimal hedge is analytically known as the \emph{delta hedge}, and we aim for the output of our method to closely approximate this optimal solution.

    \vspace{5pt}    
    We consider here a GBM with volatility $\sigma = 0.2$
    \[
    S_0 = 1, \quad dS_t = \sigma S_t dB_t
    \]
    on a time horizon of $T = 30$ days, regularly sampled $M = 120$ times.

    Using the setup described in Theorem \ref{thm:Immanol}, we compute our kernel hedge with \emph{signature kernels}, constructing the Gram matrix with entries as defined in Eqn. (\ref{eqn:gram_sig}). This is performed for training samples \(\bX_N\) of increasing sizes $N \in \{10, 50, 100, 200\}$. Subsequently, we compute the optimal kernel hedges $t \mapsto f^*_N(Y|_{[0, t]})$ on a separate, independently sampled corpus $\mathbb{Y}$ consisting of $500$ realizations of $S$.

    \vspace{5pt}
    Figure \ref{fig:pnl_distrib} shows the \emph{PnL} distributions of the kernel hedges alongside the distribution from the optimal delta hedge. The results demonstrate that as the number of samples $N$ used for fitting increases, the PnL distribution converges to that of the optimal hedge.
    
    Figure \ref{fig:hedges} illustrates the hedging positions for a specific realization $Y \in \mathbb{Y}$, highlighting how the positions converge at all times to their theoretically optimal values. We note that the convergence deteriorates over time due to the coarseness of the time grid.
    
    

    \begin{figure}[h!]  
    \centering
    \includegraphics[width=9cm]{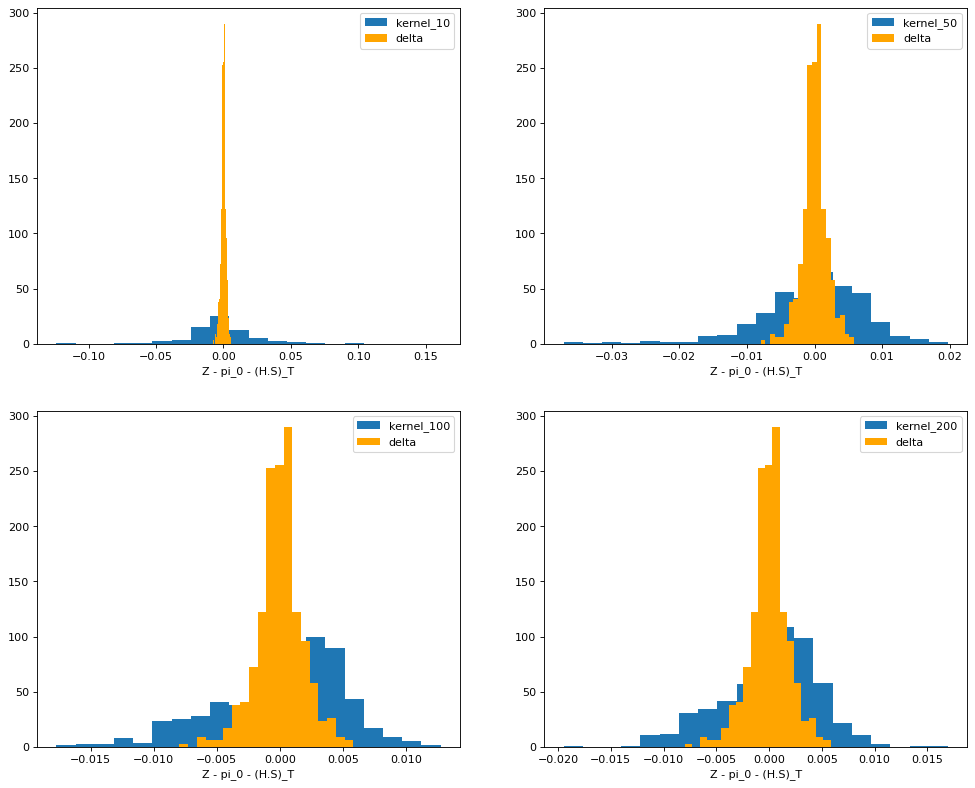}
    \caption{Distribution of PnLs with increasing training size (10, 50, 100, 200).}
    \label{fig:pnl_distrib}
    \end{figure}

    \begin{figure}[h!]  
    \centering
    \includegraphics[width=9cm]{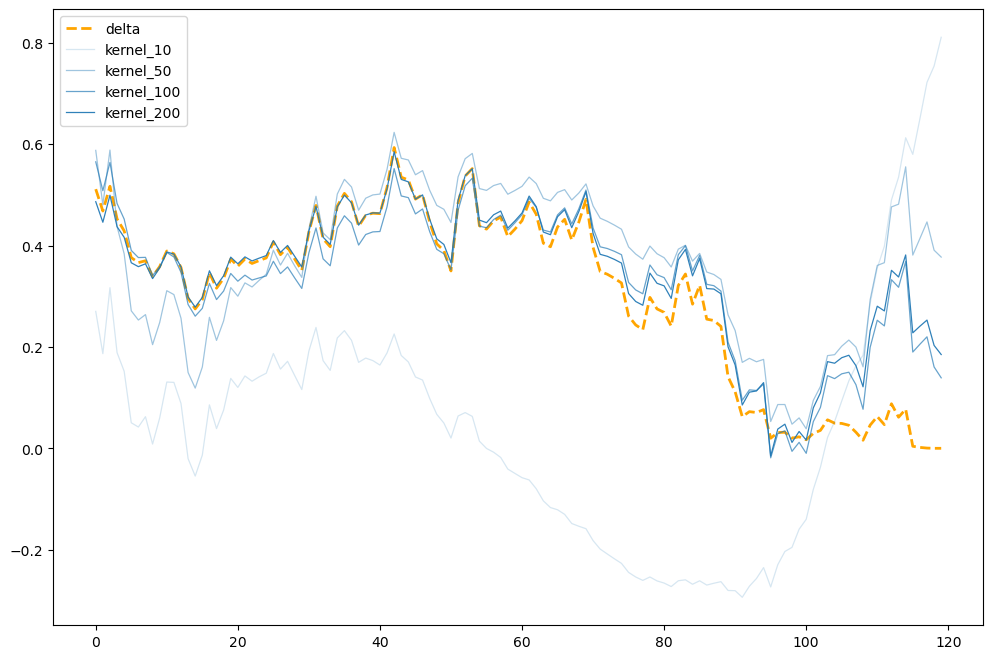}
    \caption{Positions for training sizes (10, 50, 100, 200).}
    \label{fig:hedges}
    \end{figure}

\section{Conclusion}

In conclusion, this work introduces a robust, signature-based hedging algorithm that leverages operator-valued kernels and un-truncated signature kernels to address the limitations of existing methods in high-dimensional financial markets. By employing minimal assumptions and modelling market dynamics as general geometric rough paths, our framework remains entirely model-free. We establish a representer theorem that guarantees a unique global minimizer, and we derive an analytic solution for a broad class of loss functions. This approach can incorporate any relevant auxiliary data—ranging from trading signals to market news or prior decisions—via the operator-valued kernel, thereby mirroring real-world trading practices and fostering more informed, machine-learning-driven hedging strategies.

\newpage

\appendix

\section{Rough Paths Primer}
\label{app:RP_Primer}

We here give a brief overview of Rough Paths theory, citing without proof the results we are using for our proofs.
We follow \cite{friz2020course, chevyrev_rough_2024, lyons2002system}, to which we refer the reader for further details.

Rough paths theory was developed by Lyons in the 1990's to make path-wise sense of equations of type
\begin{equation}\label{eqn:generic_dY}
    \int_0^t F(Y_s)dX_s \in \bR^N
\end{equation}
where $X: [0,1] \to \bR^d$ is a $\alpha$-Holder path with $\alpha \leq \frac{1}{2}$.

The main insight of the theory is that the information contained in $X$ alone is not sufficient to guarantee existence and uniqueness of solutions: additional \say{higher order} information is required, intuitively corresponding to \say{deep-enough} (up to $M := \lfloor{\frac{1}{\alpha} \rfloor}$) iterated integrals
\begin{equation}\label{eqn:iterated_ints}
    \int\limits_{s \leq r_1 \leq \cdots r_m \leq t} dX_{r_1} \otimes \cdots \otimes dX_{r_m} \in (\bR^d)^{\otimes m}.
\end{equation}

Think about the case where $X$ is a semi-martingale, is (\ref{eqn:generic_dY}) to be understood in It\^o or Stratonovich sense? These two notions of stochastic integration differ, and this difference is quantified by a second-order objet: the quadratic variation of $X$.



The extra information is encoded in continuous
\footnote{According to the normed structure of the truncated tensor algebra $T^{(M)}$ (\emph{c.f.} \cite[2.4]{friz2020course}).}
maps 
$$\mathbf{X}: \{(s,t) \in [0,1]^2 ~|~ s \leq t \} =: \Delta \to T^{(M)}(\bR^d) := \bigoplus_{m=0}^{M} (\bR^d)^{\otimes m}$$
extending $X$ in the sense that $\mathbf{X}^1_{s,t} = X_t - X_s$ and behaving like iterated integrals both algebraically, meaning that Chen's relation
\begin{equation}
\label{eqn:Chen}
    \mathbf{X}_{s,t} 
    = 
    \mathbf{X}_{s,u} \otimes \mathbf{X}_{u,s}
\end{equation}
holds for any $0\leq s \leq u \leq t \leq 1$, and analytically, meaning that there exist $C > 0$ such that
\begin{equation}
\label{eqn:analytic_rp_request}
   \forall ~ |I| \leq M, ~ \forall s \leq t, ~ |\sprod{\mathbf{X}_{s,t}}{e_I}| \leq C |s-t|^{|I|\alpha}.
\end{equation}

\begin{definition}
    The space of functions $\mathbf{X}: \Delta \to T^{(M)}(\bR^d)$ satisfying Chen's relation (\ref{eqn:Chen}) and the analytic requirement (\ref{eqn:analytic_rp_request}) is the space of \emph{$\alpha$-H\"older rough paths}, denoted by $\mathscr{C}^{\alpha}([0,1]; \bR^d)$.
\end{definition}

While the set of $\alpha$-H\"older rough paths does not have a vector-space structure, it can be endowed with a distance, and hence constitutes a metric space (\emph{c.f.} \cite[2.4]{friz2020course}).

\begin{remark}
    Every smooth path $X \in C^{\infty}([0,1];\bR^d)$ can be canonically lifted to an 
    element of $\mathscr{C}^{\alpha}([0,1]; \bR^d)$ for any $\alpha \in (0,1)$, by
    defining the $\mathbf{X}_{s,t}^m$ as the iterated integrals (\ref{eqn:iterated_ints}), 
    always defined under the smoothness hypotheses. 
    The space of such lifts is $\mathscr{L}^{\infty}(C^{\infty}) \subset \mathscr{C}^{\alpha}$.
\end{remark}

\begin{definition}
    The closure of $\mathscr{L}^{\infty}(C^{\infty})$ in $\mathscr{C}^{\alpha}([0,1]; \bR^d)$, is denoted by $\mathscr{C}^{0, \alpha}_g([0,1]; \bR^d)$ and is called the space of 
    \emph{geometric $\alpha$-H\"older rough paths}. 
\end{definition}

Using the properties of the chain rule one can prove that geometric rough paths possess additional algebraic structure, 
in particular the following \emph{shuffle identity} holds:
\begin{equation}\label{eqn:shuffle}
    \forall \mathbf{X} \in \mathscr{C}^{0, \alpha}_g([0,1]; \bR^d), ~ \forall I,J, ~  \sprod{\mathbf{X}}{e_I}\sprod{\mathbf{X}}{e_J} = \sprod{\mathbf{X}}{e_I \shuffle e_J}
\end{equation}

In our setting we will need to give meaning to simpler integrals than (\ref{eqn:generic_dY}): \say{linear} integrals of type 
\[
\int_0^t Y_s dX_s.
\]

This is accomplished by the notion of \emph{controlled} Rough Path (\emph{c.f.} \cite[Def. 4.6 \& Ex. 4.18]{friz2020course}):
\begin{definition}
    Let $\alpha \in (0,1)$, $M = \lfloor{\frac{1}{\alpha}} \rfloor$ and $W$ a Banach Space. Given $\mathbf{X} \in \mathscr{C}^{\alpha}([0,1]; \bR^d)$ we define the space of \emph{controlled (by $\mathbf{X}$) $W$-valued rough paths} as the set 
    \[
    \mathscr{D}^{M\alpha}_{\mathbf{X}}([0,1]; W) \subseteq \left\{ [0,1] \to \mathcal{L}(T^{(M-1)}(\bR^d); W) \right\}
    \]
    of functions $\mathbf{Y}$ satisfying 
    \begin{equation}
        \forall |I| \leq M - 1, ~ \norm{\sprod{\mathbf{Y}_t}{e_I} - \sprod{\mathbf{Y}_s}{\mathbf{X}_{s,t} \otimes e_I}}_{W} \leq C |t-s|^{(M-|I|)\alpha}
    \end{equation}
    We say that $\mathbf{Y}$ lifts $Y_t = \mathbf{Y}_t^{()} \in W$.
\end{definition}

\begin{proposition}
Given $\mathbf{X} \in \mathscr{C}^{\alpha}([0,1]; \bR^d)$ and $\mathbf{Y} \in \mathscr{D}^{M\alpha}_{\mathbf{X}}([0,1]; W)$ it is well defined, for all $i \in 1, \dots, d$, the integral
\[
\int_0^t Y_s \D\mathbf{X}^i_s := 
\lim\limits_{|\mathcal{P}| \to 0}
\sum_{[s,t] \in \mathcal{P}}
\sum_{|I| < M} \mathbf{Y}_s^I \mathbf{X}^{Ii}_{s,t} \in W
\]
\end{proposition}

Given $W = \mathcal{L}(\bR^d; V)$, using the identification
$$\mathcal{L}((\bR^d)^{\otimes m}; W) \simeq \mathcal{L}((\bR^d)^{\otimes m + 1}; V) \text{ as } \phi(\mathbf{x})[y] \sim \phi[\mathbf{x} \otimes y], $$ 
one can then define
\[
\int_0^t Y_s[\D \mathbf{X}_s] = \sum_{i=1}^d \int_0^t Y_s[e_i] ~ \D \mathbf{X}^i_s := 
\lim\limits_{|\mathcal{P}| \to 0}
\sum_{[s,t] \in \mathcal{P}} 
\mathbf{Y}_s[\mathbf{X}_{s,t}] \in V
\]

\begin{proposition}[Thm. 3.7 \cite{lyons2007differential}]
Given $\mathbf{X} \in \mathscr{C}^{\alpha}([0,1];\bR^d)$ there exists a unique element $Sig(\mathbf{X})$ of $\Delta \to T((\bR^d))$ extending $\mathbf{X}: \Delta \to T^{\lfloor{\frac{1}{\alpha} \rfloor}}(\bR^d)$, satisfying Chen's relation
\[
    Sig(\mathbf{X})_{s,u} \otimes Sig(\mathbf{X})_{u,t} = Sig(\mathbf{X})_{s,t}
\]
and such that 
\[
\forall I, ~ \exists C_{|I|} \geq 0, ~ \forall s \leq t, ~ |\sprod{Sig(\mathbf{X})_{s,t}}{e_I}| \leq C_{|I|}|t-s|^{|I|\alpha} 
\]

\end{proposition}


\section{Operator Valued Kernels} \label{app:operator_kernels}

Recall the definition of Operator Valued Kernels:

\begin{definition}(Operator Valued Kernel)
    Given a set $\mathcal{Z}$ and a real Hilbert space $\mathcal{Y}$. 
    An $\mathcal{L}(\mathcal{Y})$-valued kernel is a map $K: \mathcal{Z} \times   \mathcal{Z} \to \mathcal{L}(\mathcal{Y})$ such that
    \begin{enumerate}[label=\roman*.]
        \item $\forall x,z \in \mathcal{Z}$ it holds $K(x,z) = K(z,x)^T$.
        \item For every $N \in \bN$ and $\{(z_i,y_i)\}_{i=1,\dots,N}\subseteq \mathcal{Z} \times \mathcal{Y}$ the matrix with entries $\sprod{K(z_i,z_j)y_i}{y_j}_{\mathcal{Y}}$ is semi-positive definite.
    \end{enumerate}

    Recall that the RKHS $\cH_{K} \xhookrightarrow{} \{\mathcal{Z} \to \mathcal{Y}\}$ has the following properties:
    \begin{itemize}
        \item $\forall z,y \in \mathcal{Z} \times \mathcal{Y}$ it holds $K(z,\cdot)y \in \cH_{K}$
        \item $\forall z,y \in \mathcal{Z} \times \mathcal{Y}$ and $\forall \phi \in \cH_{K}$ it holds $\sprod{\phi}{K(z,\cdot)y}_{\cH_{K}} = \sprod{\phi(z)}{y}_{\mathcal{Y}}$
    \end{itemize}
\end{definition}

\section{Well Posedness in Signature Case}
\label{app:wellDefSig}

\begin{lemma}
    Let $\alpha \in (0, 1)$ and $N = \lfloor{\frac{1}{\alpha}}\rfloor$. 
    Let $\bfX, \tilde{\bfX} \in \mathscr{C}^\alpha([0,T]; \bR^d)$  coinciding but for the last entry \emph{i.e.} such that 
    $$\forall 0 \leq m \leq N - 1. \quad 
    X^m = \tilde X^m$$
    Then $Sig(\mathbf{\tilde X}) \in l^2(\mathbb{W}_d)$ is controlled by $\mathbf{X}$, thus the choice of kernel 
    $$K(X,Y) = K_{sig}(\mathbf{\tilde X},\mathbf{\tilde Y}) Id_d$$
    satisfies the assumptions of Proposition \ref{prop:Phi_rough_all} with integration against $\D \mathbf{X}$.
\end{lemma}

\begin{proof}
    From \cite[Thm. 3.1.3]{lyons2002system} we have that the signature lift satisfies 
    \[
    \norm{\sum_{|I|=n} Sig(\mathbf{\tilde X})_{s,t}^I \hspace{3pt} e_I}_{(\bR^d)^{\otimes n}} \leq  C \frac{|t-s|^{n\alpha}}{n!}
    \]
    for some constant $C$ possibly dependent on $X$ but not on $n$. Then for $m \in \mathbb{N}$
    \[
    \norm{Sig(\mathbf{\tilde X})_{s,t} - 
    \sum_{|I| \leq m} Sig(\mathbf{\tilde X})_{s,t}^I e_I
    }_{l^2(\mathbb{W}_d) } 
    \leq C |t-s|^{{(m+1)}{\alpha}} e^{|t-s|^{\alpha}} 
    \leq K |t-s|^{{(m+1)}{\alpha}}
    \]
    for some $K$ independent of $m$.
    Note how $\cH_{K} = (\cH_{Sig})^d \xhookrightarrow{} \cH_{Sig} \boxtimes \bR^d$ with $(v_1, \dots, v_d) \mapsto \sum_{i=1}^d v_i \boxtimes e_i$ and under this embedding 
    \[
    K(X)_s z = Sig(\mathbf{\tilde X})_{0,s} \boxtimes z
    \]
    then
    \[
    K(X)_t z - K(X)_s z = ( Sig(\mathbf{\tilde X})_{0,t} - Sig(\mathbf{\tilde X})_{0,s}) \boxtimes z = [ Sig(\mathbf{\tilde X})_{0,s}  \otimes (Sig(\mathbf{\tilde X})_{s,t} - 1)] \boxtimes z 
    \]
    where we have only used Chen's relation.
    Thus for some $\bar K$ independent of $s,t$
    \[
    \norm{K(X)_t - K(X)_s}_{\mathcal{L}(\bR^d, \cH_{K})} \leq \norm{Sig(\mathbf{\tilde X})_{0,s}} \norm{Sig(\mathbf{\tilde X})_{s,t} - 1} \leq \bar K |t-s|^{\alpha}
    \]

    Moreover, with an abuse of notation, we have
    \begin{align*}
        K(X)_t z 
        = & \hspace{3pt} Sig(\mathbf{\tilde X})_{0,t} \boxtimes z 
        = [ Sig(\mathbf{\tilde X})_{0,s} \otimes Sig(\mathbf{\tilde X})_{s,t}] \boxtimes z 
        \\
        = & \hspace{3pt} [ Sig(\mathbf{\tilde X})_{0,s} \otimes  Sig(\mathbf{\tilde X})_{s,t}^{< N} ] \boxtimes z +  [ Sig(\mathbf{\tilde X})_{0,s}  \otimes  Sig(\mathbf{\tilde X})_{s,t}^{\geq N} ] \boxtimes z 
    \end{align*}

    Define then $K^m(X)_s \in \mathcal{L}((\bR^d)^{\otimes m}, \mathcal{L}(\bR^d, \cH_{K}))$ for $m = 1, \cdots, N-1$ as
    \[
        ( K^m(X)_s v ) z :=  (Sig(\mathbf{\tilde X})_{0,s} \otimes  v) \boxtimes z 
    \]
    where $(v, z) \in (\bR^d)^{\otimes m} \times \bR^d$. 
    Then for all $m$ we have as needed
    \[
     \norm{K^m(X)_t - K^m(X)_s} \leq \norm{Sig(\mathbf{\tilde X})_{0,t} - Sig(\mathbf{\tilde X})_{0,s}} \leq \bar K |t-s|^{\alpha}
    \]

    Moreover 
    \begin{align*}
        & K^m(x)_t[v][z] - \sum_{k=0}^{N-1-m} K^{m+k}(x)_s [\mathbf{X}^{k}_{s,t} \otimes v] [z] =
        \\
        & (Sig(\mathbf{\tilde X})_{0,t} \otimes  v) \boxtimes z 
        - \sum_{k=0}^{N-1-m} (Sig(\mathbf{\tilde X})_{0,s} \otimes \mathbf{X}^{k}_{s,t} \otimes  v) \boxtimes z = 
        \\
        & (Sig(\mathbf{\tilde X})_{0,t} \otimes  v) \boxtimes z 
        - \left ( Sig(\mathbf{\tilde X})_{0,s} \otimes  Sig(\mathbf{X})^{< N-m}_{s,t} \otimes v \right) \boxtimes z =
        \\
        & \left[ \left( Sig(\mathbf{\tilde X})_{0,t} - Sig(\mathbf{\tilde X})_{0,s} \otimes  Sig(\mathbf{X})^{< N-m}_{s,t} \right) \otimes v\right] \boxtimes z
    \end{align*}
    and since by the assumption on the two lifts $\mathbf{X}$ and $\mathbf{\tilde X}$ we have 
    \[
    Sig(\mathbf{\tilde X})_{s,t}^{< N} = Sig(\mathbf{X})_{s,t}^{< N}
    \]
    we obtain finally 
    \begin{align*}
        & K^m(x)_t[v][z] - \sum_{k=0}^{N-1-m} K^{m+k}(x)_s [\mathbf{X}^{k} \otimes v] [z] =
        \\
        & \left[ \left( Sig(\mathbf{\tilde X})_{0,t} - Sig(\mathbf{\tilde X})_{0,s} \otimes  Sig(\mathbf{\tilde X})^{< N-m}_{s,t} \right) \otimes v\right] \boxtimes z =
        \\
        & \left( Sig(\mathbf{\tilde X})_{0,s} \otimes  Sig(\mathbf{\tilde X})^{\geq N-m}_{s,t}  \otimes v\right) \boxtimes z
    \end{align*}
    from which it follows
    \[
    \norm{K^m(x)_t - \sum_{k=0}^{N-1-m} K^{m+k}(x)_s [\mathbf{X}^{k} \otimes (\cdot) ] } \leq \bar K |t-s|^{{(N-m)}{\alpha}}
    \]
    exactly as needed.
\end{proof}

\begin{remark}
    The assumption on the lifts $Sig(\mathbf{\tilde X})_{s,t}^{< N} = Sig(\mathbf{X})_{s,t}^{< N}$ is crucial for the success of the proof. 
    Note however that in case $N=2$ this assumption is trivially satisfied since the first two levels of any lift are always the constant $1$ and the underlying path itself. In particular this means that in the stochastic case with $p \in (2,3]$ we can, for example, integrate the Stratonovich Kernel against the Ito increment or vice-versa!
\end{remark}

\newpage
\section{Feature Map Existence}
\label{app:sect:feature_map}

\begin{theorem}\label{thm:generic_Phi}
    Let $\alpha \in (0, 1)$, $N = \lfloor{\frac{1}{\alpha}}\rfloor$ and $x \in \cX$.
    Let $\xi(x) \in \Omega_T^\alpha(\cH_\xi)$ be the rough path
    $$\boldsymbol{\xi}(x) = (1, {\xi}^1(x),\cdots,\xi^N(x)): \Delta_{0,T} \to T^{(N)}(\cH_{\xi})$$
    and assume the existence of a path 
    \[
    \boldsymbol{K}(x) = (K^0(x), \cdots, K^{N-1}(x)) : [0,T] \to \bigoplus\limits_{m=0}^{N-1} \mathcal{L}(\cH_{\xi}^{\otimes m}; \mathcal{L}(\cH_{\xi}; \cH_{K})) 
    \]
    extending $K^0(x)_s := K(\psi(x)|_{[0,s]}, \cdot)$ and such that for every $m \leq N - 1$
    \begin{equation}
        \norm{K^m(x)_t - \boldsymbol{K}(x)_s[\boldsymbol{\xi}(x)_{s,t} \otimes (\cdot)]}_{\mathcal{L}(\cH_{\xi}^{\otimes m}; \mathcal{L}(\cH_{\xi}; \cH_{K})) } \leq C |t-s|^{(N-m)\frac{1}{p}}
    \end{equation}
    \begin{equation}
        \norm{K^m(x)_t - K^m(x)_s}_{\mathcal{L}(\cH_{\xi}^{\otimes m}; \mathcal{L}(\cH_{\xi}; \cH_{K})) } \leq C |t-s|^{\frac{1}{p}}
    \end{equation}
    then it is well defined in a rough sense\footnote{we implicity use the inclusion $\mathcal{L}(\cH_{\xi}^{\otimes m}; \mathcal{L}(\cH_{\xi}; \cH_{K}))  \xhookrightarrow{} \mathcal{L}(\cH_{\xi}^{\otimes m +1}; \cH_{K})$} 
    \begin{align*}
        \Phi_x & := \int_0^T K(\psi(x)|_{[0,t]}, \cdot) \D \boldsymbol{\xi}(x)_t 
        \\
        & := \lim\limits_{|\mathcal{P}| \to 0} \sum\limits_{[s,t]\in \mathcal{P}}
        \sum_{m = 0}^{N-1}  K^m(x)_s [\boldsymbol{\xi}(x)^{m+1}_{s,t}]
        \in \cH_{K}.
    \end{align*}
    Furthermore each $t \mapsto f(\psi(x)|_{[0,t]})$ is itself integrable against $\D \boldsymbol{\xi}$ in a rough sense and
    \begin{equation}\label{eqn:core_swap_rough}
        \int_0^T \left\langle f(\psi(x)|_{[0,t]}), \D \boldsymbol{\xi}(x)_t \right\rangle_{\cH_{\xi}} = \left\langle f, \Phi_x \right\rangle_{\cH_{K}}
    \end{equation}
\end{theorem}

\begin{proof}
    The well posedness of $\Phi_X$ is a consequence of \cite[Sect. 4.5]{friz2020course}.
    We use the notation of \cite[A.1]{roughFubini}. 
    Regarding the integrability of RKHS elements, note how for any $z \in \cH_{K}$
    \begin{align*}
        \sprod{f(\psi(x)|_{[0,t]})}{z}_{\cH_{\xi}} 
        = &  \hspace{5pt}
        \sprod{f}{K(\psi(x)|_{[0,t]}, \cdot)z}_{\cH_{K}}
        \\
        = & \hspace{5pt}
        \sprod{f}{K^0(x)_t[1][z]}_{\cH_{K}}
        \\
        = & \hspace{5pt}
        \sprod{f}{\boldsymbol{K}(x)_s[\boldsymbol{\xi}(x)_{s,t}][z] + R^{f}_{s,t}[z]}_{\cH_{K}}
    \end{align*}
    where for all times $s,t$
    \[
    \norm{R^{f}_{s,t}}_{\mathcal{L}(\cH_{\xi}; \cH_{K})} \leq C|t-s|^{N\alpha}
    \]

    This suggests to define
    \begin{equation}
        \mathbf{f}(x) = (f^0(x), \cdots, f^{N-1}(x)): [0,T] \to  \bigoplus\limits_{m=0}^{N-1} \mathcal{L}(\cH_{\xi}^{\otimes m}; \cH^*_{\xi}) 
    \end{equation}
    as
    \[
    \sprod{f^m(x)_t[v]}{z}_{\cH_{\xi}} := \sprod{f}{(K^m(x)_t[v])[z]}_{\cH_{K}}
    \]
    where $(v,z) \in \cH_{\xi}^{\otimes m} \times \cH_{\xi}$.
    Note then how for $m \leq N - 1$ and $v \in \cH_{\xi}^{\otimes m}$ one has
    \begin{align*}
        \norm{f^m(x)_t[v] - \boldsymbol{f}(x)_s[\boldsymbol{\xi}(x)_{s,t} \otimes v]}_{\cH_{\xi}^*} 
        = & \sup_{\norm{z} \leq 1} \sprod{f^m(x)_t[v] - \boldsymbol{f}(x)_s[\boldsymbol{\xi}(x)_{s,t} \otimes v]]}{z}_{\cH_{\xi}}
        \\
        = & \sup_{\norm{z} \leq 1} \sprod{f}{{K}^m(x)_t[v][z] - \boldsymbol{K}(x)_s[\boldsymbol{\xi}(x)_{s,t} \otimes v] [z]}_{\cH_{K}}
        \\
        \leq & \norm{f}_{\cH_{K}} \norm{{K}^m(x)_t[v] - \boldsymbol{K}(x)_s[\boldsymbol{\xi}(x)_{s,t} \otimes v ]}_{\mathcal{L}(\cH_{\xi};\cH_{K})}
    \end{align*}
    hence
    \begin{align*}
        \norm{f^m(x)_t - \boldsymbol{f}(x)_s[\boldsymbol{\xi}(x)_{s,t} \otimes (\cdot)]}_{\mathcal{L}(\cH_{\xi}^{\otimes m}; \cH_{\xi}^*)} 
        \leq & 
        \norm{f}_{\cH_{K}} \norm{{K}^m(x)_t - \boldsymbol{K}(x)_s[\boldsymbol{\xi}(x)_{s,t} \otimes (\cdot)]}_{\mathcal{L}(\cH_{\xi}^{\otimes m}; \mathcal{L}(\cH_{\xi};\cH_{K}))}
        \\
        \leq & \hspace{3pt} \norm{f}_{\cH_{K}} C |t-s|^{(N-m)\alpha}
    \end{align*}

    Similarly
    \begin{align*}
        \norm{f^m(x)_t[v] - f^m(x)_s[v]}_{\cH_{\xi}^*}
        = & 
        \sup_{\norm{z} \leq 1} 
        \sprod{f^m(x)_t[v] - f^m(x)_s[v]}{z}_{\cH_{\xi}}
        \\
        = & 
        \sup_{\norm{z} \leq 1} 
        \sprod{f}{(K^m(x)_t - K^m(x)_s)[v][z]}_{\cH_{K}}
        \\
        \leq & 
        \norm{f}_{\cH_{K}} \norm{K^m(x)_t - K^m(x)_s} \norm{v}_{\cH_{\xi}^{\otimes m}}
        \\
        \leq & 
        \norm{f}_{\cH_{K}} C |t-s|^{\alpha} \norm{v}_{\cH_{\xi}^{\otimes m}}
    \end{align*}
    thus similarly to before we get the bound
    \[
    \norm{f^m(x)_t - f^m(x)_s}_{\mathcal{L}(\cH_{\xi}^{\otimes m}; \cH_{\xi}^*)} \leq \norm{f}_{\cH_{K}} C |t-s|^{\alpha} 
    \]
    
    From these bounds one gets that the rough integral is well defined and, using the inclusions $\mathcal{L}(\cH_{\xi}^{\otimes m}; \mathcal{L}(V; W))\xhookrightarrow{} \mathcal{L}(\cH_{\xi}^{\otimes m} \otimes V; W)$ and $\mathcal{L}(\cH_{\xi};\bR) = \cH_{\xi}^*$, we have
    \begin{align*}
        \int_0^T \left\langle f(\psi(x)|_{[0,t]}), d \boldsymbol{\xi}(x)_t \right\rangle_{\cH_{\xi}} 
        := &  
        \lim\limits_{|\mathcal{P}| \to 0} \sum\limits_{[s,t]\in \mathcal{P}} 
        \sum_{m=0}^{N-1} f^m(x)_s[\boldsymbol{\xi}(x)_{s,t}^{m+1}]
        \\
        = &  
        \lim\limits_{|\mathcal{P}| \to 0} \sum\limits_{[s,t]\in \mathcal{P}} \sum_{m=0}^{N-1}
        \sprod{f}{K^m(x)_s[\boldsymbol{\xi}(x)_{s,t}^{m+1}]}_{\cH_{K}}
        \\
        = &  
        \lim\limits_{|\mathcal{P}| \to 0} 
        \sprod{f}{
        \sum\limits_{[s,t]\in \mathcal{P}} \sum_{m=0}^{N-1}
        K^m(x)_s[\boldsymbol{\xi}(x)_{s,t}^{m+1}]}_{\cH_{K}}
        \\ = & \hspace{5pt}
        \sprod{f}{\Phi_x}_{\cH_{K}}
    \end{align*}
\end{proof}

{

For the sake of completeness, we establish the existence of the feature map in the semi-martingale setting, relying solely on arguments from stochastic analysis:

\begin{lemma}
    Assume $X: \Omega \to C^0([0,T]; \bR^d)$ semi-martingale under $\bP$ and $\cH_{K}$ separable. 
    If the path
    \begin{equation}
        [0,T] \ni s \mapsto K(X|_{[0,s]}, \cdot) \in \cL(\bR^d, \cH_{K}) 
    \end{equation}
    is \emph{progressively measurable} with 
    \begin{equation}
        \E\left[ \int_0^T Tr(K(X|_{0,t}, X|_{0,t})) Tr(d[X,X]_t) \right] < \infty
    \end{equation}    
    then it is well defined 
    \[
    \Phi_X := \int_0^T K(X|_{[0,t]}, \cdot) \D X_t \in L^2(\Omega;\cH_{K})
    \]
    such that for any $f \in \cH_{K}$, in $L^2(\Omega;\bR)$ sense, holds
    \begin{equation}\label{eqn:core_swap}
        \int_0^T \left\langle f(X|_{[0,t]}), \D X_t \right\rangle_{\bR^d} = \left\langle f, \Phi_X \right\rangle_{\cH_{K}}
    \end{equation}
\end{lemma}

\begin{proof}
    Note how for any $(x,z) \in \Lambda^\alpha\times \bR^d$
    \[
        \norm{K(x,\cdot)z}_{\cH_{K}}^2 = \sprod{z}{K(x,x)z} \leq \norm{z} \norm{K(x,x)z} \leq  \norm{z}_{\bR^d}^2 \norm{K(x,x)}_{op} \leq \norm{z}_{\bR^d}^2 Tr(K(x,x))
    \]
    since $K(x,x)$ is positive semidefinite, hence 
    \[
    \norm{K(x,\cdot)}_{\mathcal{L}(\bR^d;\cH_{K})}^2 \leq Tr(K(x,x))
    \]

    It follows that $\Phi_X$ is well defined in $L^2(\Omega;\cH_{K})$ as a stochastic integral.
    
    For any $f \in \cH_{K}$ the map $s \mapsto f(X|_{0,s})$ is clearly progressively measurable and the integral
    \[
    \int_0^T \left\langle f(X|_{[0,t]}), \D X_t \right\rangle_{\bR^d}
    \]
    is well defined since
    \[
    \norm{f(X|_{0,t})}^2_{\bR^d} \leq \norm{f}_{\cH_{K}}^2 Tr(K(X|_{0,t}, X|_{0,t}))
    \]
    hence
    \[
    \E\left[ \int_0^T \norm{f(X|_{0,t})}^2_{\bR^d} Tr(d[X,X]_t) \right]
    \leq
    \norm{f}_{\cH_{K}}^2 E\left[ \int_0^T Tr(K(X|_{0,t}, X|_{0,t})) Tr(d[X,X]_t) \right]
    \]
    thus the integral 
    \[
    \int_0^T \left\langle f(X|_{[0,t]}), \D X_t \right\rangle_{\bR^d} \in L^2(\Omega;\bR)
    \]
    is well defined.
    By the commutativity of linear functionals and stochastic integration 
    \begin{align*}
        \int_0^T \left\langle f(X|_{[0,t]}), \D X_t \right\rangle_{\bR^d}
        = &  \int_0^T \sprod{f}{K(X|_{0,t},\cdot)\D X_t}
        \\
        =&  \sprod{f}{\int_0^T K(X|_{0,t},\cdot)\D X_t} = \sprod{f}{\Phi_X}
    \end{align*}
    
\end{proof}

}

\newpage

\section{It\^o and Stratonovich Signatures}

\begin{definition}
Given a continuous semimartingale $X : [0,1] \to \R^d$ define the It\^o and Stratonovich Signatures as, respectively, the solutions in $T((\R^d))$ of
\begin{equation}
    S^{It\hat{o}}(X)_t = 1 + \int_0^t S^{It\hat{o}}(X)_s \otimes dX_s
\end{equation}
\begin{equation}
    S^{Str}(X)_t = 1 + \int_0^t S^{Str}(X)_s \otimes \circ dX_s
\end{equation}
\end{definition}

\begin{remark}
These equations need to be understood in coordinates i.e. for a word $I \in \W_d$ and a letter $j\in\{1,\dots,d\}$ one has
\begin{equation}
    S^{It\hat{o}}(X)^{Ij}_t = 1 + \int_0^t S^{It\hat{o}}(X)^{I}_s dX^j_s
\end{equation}
\begin{equation}
    S^{Str}(X)^{Ij}_t = 1 + \int_0^t S^{Str}(X)^{I}_s \circ dX^j_s
\end{equation}
\end{remark}

\begin{proposition}
Given a continuous semimartingale $X : [0,1] \to \R^d$ it holds
\begin{equation}
    S^{It\hat{o}}(X)_t = 1 + \int_0^t S^{It\hat{o}}(X)_s \otimes (\circ dX_s - \frac{1}{2}d[X,X]_s)
\end{equation}
where $[X,X] \in C^{1-var}([0,1];(\R^d)^{\otimes 2})$ is the quadratic variation process of $X$.
\end{proposition}

\begin{proof}
We want to prove that for any word $I \in \W_d$  and letters $j,k \in \{1,\dots,d\}$ it holds
\begin{equation}
    S^{It\hat{o}}(X)^{Ijk}_t = \int_0^t S^{It\hat{o}}(X)^{Ij}_s \circ dX^k_s - \frac{1}{2} \int_0^t S^{It\hat{o}}(X)^{I}_s d[X^j,X^k]_s
\end{equation}
By definition we have 
\[
S^{It\hat{o}}(X)^{Ijk}_t =  \int_0^t S^{It\hat{o}}(X)^{Ij}_s dX^k_s
\]
and using the definition of Stratonovich integral we get
\[
\int_0^t S^{It\hat{o}}(X)^{Ij}_s dX^k_s = \int_0^t S^{It\hat{o}}(X)^{Ij}_s \circ dX^k_s - \frac{1}{2} [S^{It\hat{o}}(X)^{Ij},X^k]_t
\]
which gives the result noticing how 
\[
[S^{It\hat{o}}(X)^{Ij},X^k]_t 
= [\int_0^{\cdot} S^{It\hat{o}}(X)^{I}_s dX^j_s,X^k]_t
= \int_0^t S^{It\hat{o}}(X)^{I}_s d[X^j,X^k]_s
\]
\end{proof}

\newpage
\section{It\^o Signature Kernel}
\label{app:sect:ItoSigKer}

\begin{definition}
Given continuous semimartingales $X, Y : [0,1] \to \R^d$ define the It\^o and Stratonovich Signature Kernels as, respectively,
\begin{equation*}
K^{It\hat{o}}(X,Y)_{s,t} = \sprod{S^{It\hat{o}}(X)_s}{S^{It\hat{o}}(Y)_t}_{\mathbb{H}} = 1 + \sum_{k=1}^d \int_{\sigma = 0}^s \int_{\tau = 0}^t K^{It\hat{o}}(X,Y)_{\sigma,\tau} dX^k_{\sigma} dY^k_{\tau}
\end{equation*}
\begin{equation*}
    K^{Str}(X,Y)_{s,t} = \sprod{S^{Str}(X)_s}{S^{Str}(Y)_t}_{\mathbb{H}} = 1 + \sum_{k=1}^d \int_{\sigma = 0}^s \int_{\tau = 0}^t K^{Str}(X,Y)_{\sigma,\tau} \circ dX^k_{\sigma} \circ dY^k_{\tau}
\end{equation*}
\end{definition}

Here is stated the main result:

\begin{theorem}
Given continuous semimartingales $X, Y : [0,1] \to \R^d$ their It\^o Signature Kernel can be written only in terms of Stratonovich integrals as the solution of the following system of 3 SDEs:

\begin{equation*}
    \R^d \ni G(s,t) = \int_{\tau = 0}^t \int_{\sigma = 0}^s K^{It\hat{o}}(X,Y)_{\sigma,\tau} d[Y,Y]^T_{\tau} \circ dX_{\sigma} - \frac{1}{2} \int_{\tau = 0}^t 
    d[Y, Y]^T_{\tau} F(s,\tau)
\end{equation*}
\begin{equation*}
    \R^d \ni F(s,t) = \int_{\tau = 0}^t \int_{\sigma = 0}^s K^{It\hat{o}}(X,Y)_{\sigma,\tau} d[X,X]^T_{\sigma} \circ dY_{\tau} - \frac{1}{2} \int_{\sigma = 0}^s
    d[X, X]^T_{\sigma} G(\sigma, t)
\end{equation*}
\begin{align*}
   \R \ni K^{It\hat{o}}(X,Y)_{s,t} = & 
    \hspace{5pt} 1 
    +  \int_0^s \int_0^t 
    K^{It\hat{o}}(X,Y)_{\sigma,\tau}  
    \left[ 
    \sprod{\circ dX_{\sigma}}{\circ dY_{\tau}}_{\R^d}
    + \frac{1}{4} Tr(d[X,X]^T_{\sigma} d[Y,Y]_{\tau})
    \right]
    \\
    & - \frac{1}{2} 
    \int_{0}^s \sprod{G(\sigma,t)}{\circ dX_{\sigma}}_{\R^d}
    - \frac {1}{2} 
    \int_{0}^t \sprod{F(s,\tau)}{\circ dY_{\tau}}_{\R^d}
\end{align*}

\end{theorem}

\begin{proof}
Using the previous result one writes
\begin{align*}
 S^{It\hat{o}}(X)^{Ijk}_s \cdot S^{It\hat{o}}(Y)^{Ijk}_t = & 
    \int_0^s S^{It\hat{o}}(X)^{Ij}_{\sigma} \circ dX^k_{\sigma}\cdot
    {\int_0^t S^{It\hat{o}}(Y)^{Ij}_{\tau} \circ dY^k_{\tau}}
    \\
    & + \frac{1}{4} {\int_0^s S^{It\hat{o}}(X)^{I}_{\sigma} d[X^j,X^k]_{\sigma}}\cdot
    {\int_0^t S^{It\hat{o}}(Y)^{I}_{\tau} d[Y^j,Y^k]_{\tau}}
    \\
    & - \frac{1}{2} {\int_0^s S^{It\hat{o}}(X)^{I}_{\sigma} d[X^j,X^k]_{\sigma}}\cdot
    {\int_0^t S^{It\hat{o}}(Y)^{Ij}_{\tau} \circ dY^k_{\tau}}
    \\
    & - \frac{1}{2} {\int_0^s S^{It\hat{o}}(X)^{Ij}_{\sigma} \circ dX^k_{\sigma}}\cdot
    {\int_0^t S^{It\hat{o}}(Y)^{I}_{\tau} d[Y^j,Y^k]_{\tau}}
    \\
    = & 
    \int_0^s \int_0^t { S^{It\hat{o}}(X)^{Ij}_{\sigma}}
    { S^{It\hat{o}}(Y)^{Ij}_{\tau}}_{\mathbb{H}}\circ dX^k_{\sigma}\circ dY^k_{\tau}
    \\
    & + \frac{1}{4} \int_0^s \int_0^t 
    {S^{It\hat{o}}(X)^{I}_{\sigma}}
    {S^{It\hat{o}}(Y)^{I}_{\tau}}
    d[X^j,X^k]_{\sigma} d[Y^j,Y^k]_{\tau}
    \\
    & - \frac{1}{2} \int_0^s \int_0^t  
    {S^{It\hat{o}}(X)^{I}_{\sigma} d[X^j,X^k]_{\sigma}}
    {S^{It\hat{o}}(Y)^{Ij}_{\tau} \circ dY^k_{\tau}}
    \\
    & - \frac{1}{2} \int_0^s \int_0^t
    { S^{It\hat{o}}(X)^{Ij}_{\sigma} \circ dX^k_{\sigma}}
    {S^{It\hat{o}}(Y)^{I}_{\tau} d[Y^j,Y^k]_{\tau}}
\end{align*}

The only problematic terms are those of the form 
\[
\int_0^s \int_0^t  
    {S^{It\hat{o}}(X)^{I}_{\sigma} }
    {S^{It\hat{o}}(Y)^{Ij}_{\tau} d[X^j,X^k]_{\sigma} \circ dY^k_{\tau}}
\]
since one does not recover a term present in the expanded form of the kernel 
\emph{i.e.} in the expression
\begin{equation}
    K^{It\hat{o}}(X,Y)_{s,t} = \sum_{J \in \W_d}  {S^{It\hat{o}}(X)^J_s}{S^{It\hat{o}}(Y)^J_t}
\end{equation}
 One can however write 
\[
    S^{It\hat{o}}(Y)^{Ij}_{\tau} = \int_{0}^{\tau} S^{It\hat{o}}(Y)^{I}_{r} dY^i_r
\]
obtaining
\begin{align*}
    & \int_0^s \int_0^t  
    {S^{It\hat{o}}(X)^{I}_{\sigma}}
    {S^{It\hat{o}}(Y)^{Ij}_{\tau} d[X^j,X^k]_{\sigma} \circ dY^k_{\tau}} =
    \\
    & \int_0^s \int_0^t  
    {S^{It\hat{o}}(X)^{I}_{\sigma} }
    {\int_{0}^{\tau} S^{It\hat{o}}(Y)^{I}_{r} dY^j_r d[X^j,X^k]_{\sigma} \circ dY^k_{\tau}} =
    \\
    & \int_0^s \int_0^t  \int_{0}^{\tau}
    {S^{It\hat{o}}(X)^{I}_{\sigma} }
    { S^{It\hat{o}}(Y)^{I}_{r} d[X^j,X^k]_{\sigma} dY^j_r \circ dY^k_{\tau}} 
\end{align*}
Note that $dY^j_r$ is still an It\^o integral which we are going to remove later.

Summing over all words in $\W_d$ one thus obtains
\begin{align*}
    K^{It\hat{o}}(X,Y)_{s,t} = & 
    1 
    + \sum_{k=1}^d \int_0^s \int_0^t 
    K^{It\hat{o}}(X,Y)_{\sigma,\tau} 
    \circ dX^k_{\sigma} \circ dY^k_{\tau}
    \\
    & + \frac{1}{4} \sum_{j,k = 1}^d  \int_0^s \int_0^t 
    K^{It\hat{o}}(X,Y)_{\sigma,\tau} 
    d[X^j,X^k]_{\sigma} d[Y^j,Y^k]_{\tau}
    \\
    & - \frac{1}{2} \sum_{j,k = 1}^d  \int_0^s \int_0^t \int_{r = 0}^{\tau}
    K^{It\hat{o}}(X,Y)_{\sigma,r} 
    dY^j_r d[X^j,X^k]_{\sigma}  \circ dY^k_{\tau}
    \\
    & - \frac{1}{2} \sum_{j,k = 1}^d  \int_0^s \int_0^t \int_{u = 0}^{\sigma}
    K^{It\hat{o}}(X,Y)_{u,\tau} 
    dX^j_u \circ dX^k_{\sigma} d[Y^j,Y^k]_{\tau} 
\end{align*}

It remains to remove the It\^o integrals from the last two terms. 
To do this it suits to study the quantity
\begin{equation}
    \int_{r = 0}^{\tau} K^{It\hat{o}}(X,Y)_{\sigma,r}  dY^j_r 
    =
    \int_{r = 0}^{\tau} K^{It\hat{o}}(X,Y)_{\sigma,r}  \circ dY^j_r 
    - \frac{1}{2} [K^{It\hat{o}}(X,Y)_{\sigma,\cdot}, Y^j]_{\tau}
\end{equation}

Define at this point
\[
G(s,t) := ([K^{It\hat{o}}(X,Y)_{s,\cdot}, Y^j]_{t})_{j=1,\dots,d} \in \R^d
\]
and similarly
\[
F(s,t) := ([K^{It\hat{o}}(X,Y)_{\cdot,t}, X^j]_{s})_{j=1,\dots,d} \in \R^d
\]
then one has
\begin{align*}
    G(s,t)^j  = &  
    [K^{It\hat{o}}(X,Y)_{s,\cdot}, Y^j]_{t} 
    \\ = & 
    \sum_{k=1}^d 
    [\int_{\sigma = 0}^s \int_{\tau = 0}^{\cdot} K^{It\hat{o}}(X,Y)_{\sigma,\tau} dX^k_{\sigma} dY^k_{\tau}, Y^j]_{t} 
    \\ = & 
    \sum_{k=1}^d  
    \int_{\tau = 0}^t \int_{\sigma = 0}^s  K^{It\hat{o}}(X,Y)_{\sigma,\tau} dX^k_{\sigma} 
    d[Y^k, Y^j]_{\tau} 
    \\ = & 
    \sum_{k=1}^d  
    \int_{\tau = 0}^t \int_{\sigma = 0}^s  K^{It\hat{o}}(X,Y)_{\sigma,\tau} \circ dX^k_{\sigma} 
    d[Y^k, Y^j]_{\tau} 
    \\ & - \frac{1}{2}
    \sum_{k=1}^d  \int_{\tau = 0}^t [K^{It\hat{o}}(X,Y)_{\cdot,\tau}, X^k]_s
    d[Y^k, Y^j]_{\tau} 
    \\ = & 
    \sum_{k=1}^d  
    \int_{\tau = 0}^t \int_{\sigma = 0}^s  K^{It\hat{o}}(X,Y)_{\sigma,\tau} \circ dX^k_{\sigma} 
    d[Y^k, Y^j]_{\tau} 
    \\ & - \frac{1}{2}
    \sum_{k=1}^d  \int_{\tau = 0}^t F(s,\tau)^k
    d[Y^k, Y^j]_{\tau} 
\end{align*}
which can be written in vector form as
\begin{equation}\label{eqn:G_pde}
    G(s,t) = \int_{\tau = 0}^t \int_{\sigma = 0}^s K^{It\hat{o}}(X,Y)_{\sigma,\tau} d[Y,Y]^T_{\tau} \circ dX_{\sigma} - \frac{1}{2} \int_{\tau = 0}^t 
    d[Y, Y]^T_{\tau} F(s,\tau)
\end{equation}
Similarly one has 
\begin{equation}
    F(s,t) = \int_{\tau = 0}^t \int_{\sigma = 0}^s K^{It\hat{o}}(X,Y)_{\sigma,\tau} d[X,X]^T_{\sigma} \circ dY_{\tau} - \frac{1}{2} \int_{\sigma = 0}^s
    d[X, X]^T_{\sigma} G(\sigma, t)
\end{equation}

Rewriting in vector form the equalities regarding the Kernel one obtains 
\begin{align*}
   K^{It\hat{o}}(X,Y)_{s,t} = & 
    1 
    +  \int_0^s \int_0^t 
    K^{It\hat{o}}(X,Y)_{\sigma,\tau} 
    \sprod{\circ dX_{\sigma}}{\circ dY_{\tau}}_{\R^d}
    \\
    & + \frac{1}{4}  \int_0^s \int_0^t 
    K^{It\hat{o}}(X,Y)_{\sigma,\tau} 
    Tr(d[X,X]^T_{\sigma} d[Y,Y]_{\tau})
    \\
    & - \frac{1}{2} \int_0^s \int_0^t \int_{r = 0}^{\tau}
    K^{It\hat{o}}(X,Y)_{\sigma,r} 
    (\circ dY_r)^T d[X,X]_{\sigma}  \circ dY_{\tau}
    \\
    & + \frac{1}{4}  \int_0^s \int_0^t G(\sigma,\tau)^T d[X,X]_{\sigma}  \circ dY_{\tau}
    \\
    & - \frac{1}{2} \int_0^s \int_0^t \int_{u = 0}^{\sigma}
    K^{It\hat{o}}(X,Y)_{u,\tau} 
    ( \circ dX_u)^T d[Y,Y]_{\tau} \circ dX_{\sigma} 
    \\
    & + \frac{1}{4} \int_0^s \int_0^t F(\sigma,\tau)^T d[Y,Y]_{\tau}  \circ dX_{\sigma}
\end{align*}

Notice now how by Eqn (\ref{eqn:G_pde}) one has
\begin{align*}
    \frac{1}{2} \int_0^s \int_0^t F(\sigma,\tau)^T d[Y,Y]_{\tau}  \circ dX_{\sigma} 
    = &
    \int_{0}^s \int_{0}^t \int_{u=0}^{\sigma} K^{It\hat{o}}(X,Y)_{u,\tau} (\circ dX_{u} )^T d[Y,Y]_{\tau} \circ dX_{\sigma} 
    \\
    & - \int_{0}^s G(\sigma,t)^T \circ dX_{\sigma} 
\end{align*}
and similarly
\begin{align*}
    \frac{1}{2} \int_0^t \int_0^s G(\sigma,\tau)^T d[X,X]_{\sigma}  \circ dY_{\tau} 
    = &
    \int_{0}^t \int_{0}^s \int_{r=0}^{\tau} K^{It\hat{o}}(X,Y)_{\sigma,r} (\circ dY_{r} )^T d[X,X]_{\sigma} \circ dY_{\tau} 
    \\
    & - \int_{0}^t F(s,\tau)^T \circ dY_{\tau} 
\end{align*}
thus the triple integral terms cancel out leaving us with
\begin{align*}
   K^{It\hat{o}}(X,Y)_{s,t} = & 
    1 
    +  \int_0^s \int_0^t 
    K^{It\hat{o}}(X,Y)_{\sigma,\tau}  
    \left[ 
    \sprod{\circ dX_{\sigma}}{\circ dY_{\tau}}_{\R^d}
    + \frac{1}{4} Tr(d[X,X]^T_{\sigma} d[Y,Y]_{\tau})
    \right]
    \\
    & - \frac{1}{2} 
    \int_{0}^s G(\sigma,t)^T \circ dX_{\sigma} 
    - \frac{1}{2} 
    \int_{0}^t F(s,\tau)^T \circ dY_{\tau} 
\end{align*}

\end{proof}

\begin{remark}
    If one writes 
    \[
    \delta\X_t = 
    \begin{bmatrix}
    \circ dX^T_t 
    \\
    d[X,X]_t
    \end{bmatrix}
    \in \R^{(d+1)\times d}
    \]
    and similarly for $\delta \Y$ then one can write
    \begin{align*}
        \begin{bmatrix}
        K \\ G \\ F
        \end{bmatrix}_{s,t} 
        = 
        \begin{bmatrix}
        1 \\ 0 \\ 0
        \end{bmatrix}
        & +
        \int_0^s \int_0^t K_{\sigma,\tau} 
        \begin{bmatrix}
        Tr(\delta\Y_{\tau} \delta\X^T_{\sigma}) 
        \\ 
        2 [\delta\Y_{\tau} \delta\X^T_{\sigma}]^0_{1,\dots,d} 
        \\ 
        2 [\delta\X_{\sigma}\delta\Y^T_{\tau}]^0_{1,\dots,d}
        \end{bmatrix}
        \\
        & -
        \int_0^t 
        \begin{bmatrix}
        \frac{1}{2} [\delta \Y_{\tau}]_0
        \\ 
        [\delta \Y_{\tau}]_{1,\dots,d}
        \\ 
        0
        \end{bmatrix}
        F_{s,\tau} 
        \\
        & - 
        \int_0^s
        \begin{bmatrix}
        \frac{1}{2} [\delta \X_{\sigma}]_0
        \\ 
        0
        \\ 
        [\delta \X_{\sigma}]_{1,\dots,d}
        \end{bmatrix}
        G_{\sigma,t} 
    \end{align*}
\end{remark}

Writing moreover
\[
\mathbb{K}_{s,t} := \begin{bmatrix}
    K \\ G \\ F
    \end{bmatrix}_{s,t} 
\]

\[
\delta \mathbb{Z}_{s,t} := K_{\sigma,\tau} \begin{bmatrix}
    Tr(\delta\Y_{t} \delta\X^T_{s}) 
    \\ 
    2 [\delta\Y_{t} \delta\X^T_{s}]^0_{1,\dots,d} 
    \\ 
    2 [\delta\X_{s}\delta\Y^T_{t}]^0_{1,\dots,d}
    \end{bmatrix}
\]

\[
\delta \mathbb{F}_{s,t} = \begin{bmatrix}
    \frac{1}{2} [\delta \Y_{t}]_0
    \\ 
    [\delta \Y_{t}]_{1,\dots,d}
    \\ 
    0
    \end{bmatrix}
    F_{s,t} 
\]
and analogously for $\delta \mathbb{G}$ we obtain
\[
\mathbb{K}_{s,t} = \begin{bmatrix}
    1 \\ 0 \\ 0
    \end{bmatrix}
    +
    \int_0^s \int_0^t \delta \mathbb{Z}_{\sigma, \tau}
    -
    \int_0^t \delta \mathbb{F}_{s,\tau}
    -
    \int_0^s \delta \mathbb{G}_{\sigma, t}
\]

Note then that
\begin{align*}
     \mathbb{K}_{s,t + \delta t} + \mathbb{K}_{s + \delta s,t} - \mathbb{K}_{s,t}
    = &
    \begin{bmatrix}
    1 \\ 0 \\ 0
    \end{bmatrix} 
    +
    \begin{bmatrix}
    1 \\ 0 \\ 0
    \end{bmatrix}
    - 
    \begin{bmatrix}
    1 \\ 0 \\ 0
    \end{bmatrix}
    \\
    &
    +
    \left( \int_0^s \int_0^{t + \delta t}  +  \int_0^{s + \delta s} \int_0^{t} -  \int_0^s \int_0^{t}\right) \delta \mathbb{Z}_{\sigma, \tau}
    \\
    & -
    \left( \int_0^{t + \delta t}  -  \int_0^{t}\right)
    \delta \mathbb{F}_{s,\tau}
    -
    \int_0^{t} \delta \mathbb{F}_{s + \delta s,\tau}
    \\
    & - 
    \left(\int_0^{s + \delta s} -  \int_0^s \right)
    \delta \mathbb{G}_{\sigma,t}
    -
    \int_0^{s} \delta \mathbb{G}_{\sigma,t + \delta t}
    \\
    = &
    \begin{bmatrix}
    1 \\ 0 \\ 0
    \end{bmatrix} 
    +
    \left(\int_0^{s + \delta s}  \int_0^{t + \delta t} - \int_s^{s + \delta s}  \int_t^{t + \delta t}\right) \delta \mathbb{Z}_{\sigma, \tau}
    \\
    & -
    \int_t^{t + \delta t}
    \delta \mathbb{F}_{s,\tau}
    -
    \left( \int_0^{t + \delta t}  -  \int_t^{t + \delta t}\right)
    \delta \mathbb{F}_{s + \delta s,\tau}
    \\
    & - 
    \int_s^{s + \delta s} 
    \delta \mathbb{G}_{\sigma,t}
    -
    \left(\int_0^{s + \delta s} -  \int_s^{s + \delta s} \right)
    \delta \mathbb{G}_{\sigma,t + \delta t}
    \\
    = & \hspace{5pt} 
    \mathbb{K}_{s+\delta s, t+\delta t}
    -
    \int_s^{s + \delta s}  \int_t^{t + \delta t}  \delta \mathbb{Z}_{\sigma, \tau}
    \\
    & 
    +
    \int_t^{t + \delta t} (\delta \mathbb{F}_{s + \delta s,\tau} - \delta\mathbb{F}_{s,\tau})
    +
    \int_s^{s + \delta s}  (\delta\mathbb{G}_{\sigma,t + \delta t} - \delta\mathbb{G}_{\sigma,t}) 
\end{align*}

This gives an Euler discretization on grids of the type
\begin{align*}
    \mathbb{K}_{s_{n+1}, t_{m+1}} = & \hspace{5pt}
    \mathbb{K}_{s_{n}, t_{m+1}} + \mathbb{K}_{s_{n+1}, t_{m}} + \mathbb{K}_{s_{n}, t_{m}}
    + \Delta \mathbb{Z}_{s_n, t_m} 
    \\
    & 
    - (\Delta \mathbb{F}_{s_{n+1},t_m} - \Delta\mathbb{F}_{s_n,t_m})
    - (\Delta \mathbb{G}_{s_{n},t_{m+1}} - \Delta\mathbb{G}_{s_n,t_m})
\end{align*}


\bibliographystyle{spmpsci}  
\bibliography{references}   

\end{document}